\documentclass[letterpaper, 10 pt]{IEEEtran}  % Comment this line out if you need a4paper

\IEEEoverridecommandlockouts                              % This command is only needed if 
\usepackage{amsmath}
\usepackage{amssymb}
\usepackage{amstext}

\usepackage{graphics}

\usepackage{algorithm}
\usepackage{algorithmic}

\usepackage{color}

\usepackage{subcaption}

\usepackage{tikz}
\usetikzlibrary{arrows}
\usepackage{verbatim}
\usetikzlibrary{positioning}

\usepackage{balance}

\usepackage[mathscr]{euscript}

\usepackage[normalem]{ulem}

\newcommand{\dist}{{\rm dist}}

\newcommand{\mR}{{\mathbb R}}
\newcommand{\mC}{{\mathbb C}}
\newcommand{\mD}{{\mathbb D}}
\newcommand{\mS}{{\mathbb S}}
\newcommand{\mT}{{\mathbb T}}

\newcommand{\Ltwo}{L_2}
\newcommand{\Linf}{L_\infty}

\newcommand{\Htwo}{\mathcal{H}_2}
\newcommand{\Hinf}{\mathcal{H}_\infty}

\newcommand{\HtwoC}{\Htwo(\mC_+)}
\newcommand{\HinfC}{\Hinf(\mC_+)}
\newcommand{\quotientfield}{\mathcal{F}(\Hinf)}

\newcommand{\absfuncwzero}{\phi_{\bar{\tau}}}
\newcommand{\outerfuncwzero}{W_{\bar\tau}}
\newcommand{\outerfuncwzerotilde}{\tilde W_{\bar\tau}}
\newcommand{\cut}[1]{{\rm cut}(#1)}

\newcommand{\real}{\text{Re}}
\newcommand{\imag}{\text{Im}}

\newcommand{\taumax}{\tau_{\rm max}}

\newcommand{\script}[1]{\EuScript{#1}}

\newcommand{\qed}{\hspace*{\fill}~\IEEEQED\par}

\newcommand{\rev}[1]{\textcolor{black}{#1}}

\newtheorem{thm}{Theorem}

\newtheorem{lemma}[thm]{Lemma}
\newtheorem{prop}[thm]{Proposition}

\newtheorem{remark}{Remark}

\graphicspath{{Figures/}{Figures/Delay/}{Figures/Gain_Phase/}}

\title{
An analytic interpolation approach to stability margins with emphasis on time delay
}

\author{Axel Ringh,~\IEEEmembership{Member,~IEEE}, Johan Karlsson,~\IEEEmembership{Senior Member,~IEEE}, and Anders Lindquist,~\IEEEmembership{Life Fellow,~IEEE}% <-this % stops a space
\thanks{*This work was supported by the Swedish Research Council (VR),  grant 2014-5870, Knut and Alice Wallenberg foundation, grant KAW 2018.0349, the ACCESS Linnaeus Center at KTH, and SJTU-KTH cooperation grant.}% <-this % stops a space
\thanks{A.~Ringh is with the Department of Electronic and Computer Engineering, The Hong Kong University of Science and Technology, Clear Water Bay, Kowloon, Hong Kong, China. {\tt\small eeringh@ust.hk}}
\thanks{
J.~Karlsson is with the Division of Optimization and Systems Theory, Department of Mathematics, KTH Royal Institute of Technology, 10044 Stockholm, Sweden. {\tt\small johan.karlsson@math.kth.se}}%
\thanks{A.~Lindquist
is with the Department of Automation and the School of Mathematics, Shanghai Jiao Tong University, Shanghai 200240, China, and also with the Department of Mathematics, KTH Royal Institute of Technology, 10044 Stockholm, Sweden. {\tt\small alq@math.kth.se}}
}

\begin{document}

\maketitle
\thispagestyle{empty}
\pagestyle{empty}

%%%%%%%%%%%%%%%%%%%%%%%%%%%%%%%%%%%%%%%%%%%%%%%%%%%%%%%%%%%%%%%%%%%%%%%%%%%%%%%%
\begin{abstract}
Unlike the situation with gain and phase margins in robust stabilization, the problem to determine an exact maximum delay margin is still an open problem, although extensive work has been done to establish upper and lower bounds. The problem is that the corresponding constraints in the Nyquist plot are frequency dependent, and encircling the point $s=-1$ has to be done at sufficiently low frequencies, as the possibility to do so closes at higher frequencies. In this paper we present a new method for determining a sharper lower bound by introducing a frequency-dependent shift. The problem of finding such a bound simultaneously with gain and phase margin constraints is also considered. In all these problems we take an analytic interpolation approach.
\end{abstract}

%%%%%%%%%%%%%%%%%%%%%%%%%%%%%%%%%%%%%%%%%%%%%%%%%%%%%%%%%%%%%%%%%%%%%%%%%%%%%%%%
\section{Introduction}

Stability margins are essentially metrics to tell how close a control system is to instability. In robust control design of linear time invariant (LTI) systems it is important to know how much a system, stabilized by feedback, can be perturbed so that  there is still an LTI controller that stabilizes the closed loop system. The manner in which the system is perturbed will lead to different types of stability margins, the most common of which are the gain margin and the phase margin. The problems to determine the maximum gain and phase margins are completely solved  \cite{doyle1992feedback,khargonekar1985non}.

The situation when the perturbation is a time delay is much more delicate and determining the {\em maximum delay margin} has remained an open problem. That is, to determine the largest time delay $\taumax$ such that for any $\bar\tau<\taumax$ there exists an LTI controller that stabilizes the time delay system for each delay in the interval $[0,\bar\tau]$.
\rev{However, removing the restriction that the controller be LTI, methods have been developed to design stabilizing controllers for any predefined, arbitrary large upper bound on the delay \cite{miller2005stabilization}, and for any predefined, arbitrarily large simultaneous gain and delay margin \cite{gaudette2014stabilizing}.}

Time delays are common in LTI systems and have been the subject of much study in systems and control; see, e.g.,  \cite{gu2003stability, michiels2007stability}, \cite{fridman2014introduction} and references therein. They may occur through communication delay, computational delay or physical transport delay.
In \cite{middleton2007achievable}  upper bounds of the maximum delay margin are presented  for some simple systems, but in general they are not tight. Methods for finding lower bounds have been proposed, e.g., by using robust control \cite{wang1994representation, huang2000robust}, integral quadratic constrains \cite{kao2007stability} (see also \cite{megretski1997system}), and analytic interpolation \cite{qi2014fundamental, qi2017fundamental,ringh2018lower}.

In the present paper we propose methods that builds on the analytic interpolation approaches in  \cite{qi2014fundamental, qi2017fundamental,ringh2018lower} for obtaining sharper lower bounds of the maximum delay margin. We do this in the context of the corresponding gain and phase margins, also considering the problem where maximum values of the gain and phase magins are  prescribed. 
\rev{More precisely,}
in Section~\ref{sec:delay_margin} we introduce the gain, phase and delay margins and show how they are related by illustrating them in the Nyquist plot \rev{and in the corresponding plot for the complementary sensitivity function}. In Section~\ref{sec:lower_bound_basic} we introduce the analytic interpolation approach in a preliminary form first presented in \cite{ringh2018lower}, modifying the method of  \cite{qi2014fundamental, qi2017fundamental}. Then we interpret it in the Nyquist setting. Section~\ref{sec:improved_method} is devoted to an improved method for finding a lower bound using frequency dependent shift. The controller design \rev{problem} is discussed in Section~\ref{sec:control_implementation}, and in Section~\ref{sec:homotopy} we introduce a heuristic for selecting the frequency dependent shift.
In Section~\ref{sec:sim_margins} the corresponding
optimization
problem for multiple stability margins is considered,
and in Section~\ref{sec:example} we provide some numerical examples. Finally, in Section~\ref{sec:conclusions} we provide some conclusions.

\section{Stability margins for linear systems}
\label{sec:delay_margin}

Consider the feedback control system depicted in Figure~\ref{fig:blockdiagram}, where $P$ is a transfer function of a continuous-time, finite-dimensional, single-input-single-output LTI system, $K$ is a feedback controller, and $\Delta$ is a potential perturbation of the system. The robust stabilization problem is to find a controller $K$ that stabilizes the feedback interconnection for all $\Delta$ in a prescribed class of functions $\Omega$. Moreover, this controller has to belong to the class 
\[
\quotientfield := \left\{ \frac{N(s)}{D(s)} \; \Big| \; N,D \in \HinfC \text{ and } D(s) \not \equiv 0 \right\},
\]
where $\mC_+$ denotes the open right half plane, and $\HinfC$ denote the Hardy space of bounded analytic functions on $\mC_+$; see, e.g., \cite{foias1996robust}. Also note that for $F \in \HinfC$ we have the norm $\| F(s) \|_{\Hinf} := \sup_{s \in \mC_+} |F(s)| = \sup_{\omega \in \mR} | F(i\omega) | =: \| F(i\omega) \|_{\Linf}$, see, e.g., \cite{foias1996robust, hoffman1962banach}.

Let us first consider the standard problem without a perturbation, i.e., where $\Delta(s) \equiv 1$. Then the closed loop system is input-output stable if
\begin{equation}
\label{eq:stability}
1 + P(s)K(s) \neq 0 \quad \text{for all $s \in \bar{\mC}_+$},
\end{equation}
where $\bar{\mC}_+$ is the closed right half plane, including $\infty$, cf. \cite{youla1974single, khargonekar1985non}. This is equivalent to that the complementary sensitivity function
\begin{equation}\label{eq:def_T}
T(s) := P(s)K(s) \big(1 + P(s)K(s) \big)^{-1}
\end{equation}
belongs to $\Hinf$ \cite{doyle1992feedback}.  The feedback system is {\em internally stable\/} if, in addition, there is no pole-zero cancellation between $P$ and $K$ in $\bar{\mC}_+$ \cite[pp.~36-37]{doyle1992feedback}, \cite[p.~13]{helton1998classical}. For simplicity we assume that the poles and the zeros are distinct, in which case the absence of pole-zero cancellations is equivalent to the interpolation conditions
\begin{subequations}\label{eq:interpolation}
\begin{align}
  &  T(p_j) = 1,\quad j = 1,\ldots, n , \\
  &  T(z_j)= 0,\quad j = 1,\ldots, m,
\end{align}
\end{subequations}
where $p_1, \ldots, p_n$  are the unstable poles and $z_1, \ldots, z_m$ the nonminimum phase zeros of $P$, respectively; see, e.g., \cite{youla1974single}, \cite[Ch.~2 and 7]{helton1998classical}. In the sequel we shall simply say that $K$ stabilizes $P$ when all these conditions are satisfied. If the poles and zeros are not distinct the interpolation conditions need to be imposed with multiplicity \cite{youla1974single}.
\begin{figure}[tb]
\tikzstyle{int}=[draw, minimum size=2em]
\tikzstyle{init} = [pin edge={to-,thin,black}]
\tikzstyle{block} = [draw, rectangle, 
    minimum height=3em, minimum width=6em]
\tikzstyle{sum} = [draw, circle, node distance=1cm]
\tikzstyle{input} = [coordinate]
\tikzstyle{output} = [coordinate]
\tikzstyle{pinstyle} = [pin edge={to-,thin,black}]

% The block diagram code is probably more verbose than necessary
\begin{tikzpicture}[auto, node distance=2cm,>=latex']
    % We start by placing the blocks
    \node [input, name=input] {};
    \node [sum, right of=input] (sum) {};
    \node [block, right of=sum] (controller) {$K(s)$};
    \node [block, right of=controller, node distance=3cm] (system) {$P(s)$};

    % We draw an edge between the controller and system block to 
    % calculate the coordinate u. We need it to place the measurement block. 
    \draw [->] (controller) -- node[name=u] {$u$} (system);
    \node [output, right of=system] (output) {};
    \node [block, below of=u] (delay) {$\Delta(s)$};

    % Once the nodes are placed, connecting them is easy. 
    \draw [draw,->] (input) -- node {$r$} (sum);
    \draw [->] (sum) -- node {$e$} (controller);
    \draw [->] (system) -- node [name=y] {$y$}(output);
    \draw [->] (y) |- (delay);
    \draw [->] (delay) -| node[pos=0.99] {$-$} (sum);
%    \draw [->] (delay) -| node[pos=0.99] {$-$} 
%        node [near end] {$w$} (sum);
\end{tikzpicture}
\caption{Block diagram representation of a causal LTI SISO feedback interconnection between a controller $K$, a plant $P$, and an uncertainty $\Delta$.}
\label{fig:blockdiagram}
\end{figure}
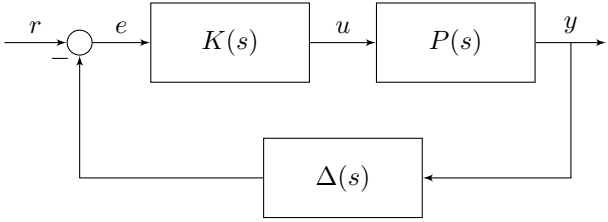

If $K$ stabilizes $P$, by continuity it also stabilizes $P\Delta$ for $\Delta$ sufficiently close to 1. An important question is how large the set of perturbations can be while retaining internal stability. In the following subsections we will discuss this problem for gain, phase, and delay uncertainties.

\subsection{The gain and phase margin problems}
\label{subsec:gain_and_phase_margin}

Two classical setups are when the class of perturbations are of the form $\Delta(s) \equiv \kappa$ for $\kappa \in [1, k]$, and $\Delta(s) \equiv e^{-i\theta}$ for $\theta \in [-\varphi, \varphi]$, corresponding to an uncertainty in the gain and the phase, respectively.

For a given controller $K$, the {\em gain margin\/} is the largest $k$ for which the controller stabilizes $\kappa P$ for all $\kappa \in [1, k)$, i.e.,
\[
\sup_{k \geq 1} \; k  \text{ such that } K \text{ stabilizes } \kappa P \text{ for } \kappa \in [1, k].
\]
This $k$ is easily found in the Nyquist plot, and corresponds to the largest interval $[-1, -1/k)$ that does not intersect the Nyquist curve $K(i\omega) P(i\omega)$, cf. \cite[Sec.~9.3]{astrom2008feedback}, \cite[pp.~51-52]{doyle1992feedback}, \cite[Sec.~3.1.1]{sepulchre1997constructive}. For a given plant $P$, the maximum gain margin problem is to determine a tight upper bound for the largest achievable gain margin:
\[
\sup_{\substack{k \geq 1 \\ K \in \quotientfield}} \; k \text{ such that } K \text{ stabilizes } \kappa P \text{ for } \kappa \in [1, k].
\]
We call this the {\em maximum gain margin\/} of the plant. In the Nyquist plot, this amounts to introducing a forbidden area  $[-1, -1/k)$ where the Nyquist curve is not allowed to enter, and finding the supremum over all such areas where stabilization is still possible. This forbidden area is illustrated in Figure~\ref{subfig:nyquist_gain_margin}, \rev{where
%\red{deleate "we by"}
$\real(\cdot)$ and $\imag(\cdot)$ denote the real and imaginary parts, respectively}. Next we note that the the M{\"o}bius transformation
\begin{equation}\label{eq:mobius}
s/(1 + s)
\end{equation}
maps the Nyquist curve $P(i\omega)K(i\omega)$ to the complementary sensitivity $T(i\omega)$. This means that the forbidden area in the Nyquist plot can be understood as a forbidden area for the complementary sensitivity function, as illustrated in Figure~\ref{subfig:interpolation_gain}. More precisely, the problem of finding a controller with gain margin $k$ can be formulated as finding a complementary sensitivity function $T \in \HinfC$ such that $T$ satisfies the interpolation conditions \eqref{eq:interpolation} and $T(i\omega) \not \in [-\infty, -1/(k - 1)]$. This formulation in terms of the complementary sensitivity function can, via a conformal mapping that maps the allowed region to the (open) unit disc, be used to compute the maximum gain margin \cite{tannenbaum1980feedback, khargonekar1985non}, \cite[Sec.~11.3]{doyle1992feedback}.

Likewise, the {\em maximum phase margin\/} for a plant $P$ is the tight upper bound for the largest achievable phase margin:
\[
\sup_{\substack{\varphi \geq 0 \\ K \in \quotientfield}} \; \varphi \text{ such that } K \text{ stabilizes } e^{-i\theta} P \text{ for } \theta \in [-\varphi, \varphi].
\]
In the Nyquist plot, this corresponds to the forbidden area being the arc $\{ -e^{i\theta} \mid \theta \in [-\varphi, \varphi] \}$, as shown in Figure~\ref{subfig:nyquist_phase_margin}, cf. \cite[Sec.~9.3]{astrom2008feedback}, \cite[p.~53]{doyle1992feedback}, \cite[Sec.~3.1.1]{sepulchre1997constructive}. When the forbidden area for the Nyquist curve is mapped via \eqref{eq:mobius}, the corresponding forbidden area for the complementary sensitivity $T(i\omega)$ is given by $\cut{-\varphi} \cup \cut{\varphi}$, where we define
\begin{equation}\label{eq:def_cut}
\cut{\varphi} \! := \!
\begin{cases}
\emptyset, &\!\!\! \text{for } \varphi \! = \! 0 \\
\frac{1}{2} + \frac{i}{2} [-\infty, -\cot(\varphi/2)], &\!\!\! \text{for } 0 \! < \! \varphi \! < \! 2\pi \\
\frac{1}{2} + \frac{i}{2} [-\cot(\varphi/2), \infty], &\!\!\! \text{for } -2\pi \! < \! \varphi \! < \! 0 \\
\frac{1}{2} + i\mR, &\!\!\! \text{for }  |\varphi| \! \geq \! 2\pi.
\end{cases}
\end{equation}
This area is illustrated in \ref{subfig:interpolation_phase}. Note that the M\"obius transformation \eqref{eq:mobius} maps the unit circle to the line $0.5 + i\mR$. As in the gain margin setting, the maximum phase margin can be computed explicitly \cite[Sec.~11.4]{doyle1992feedback}.

The gain and phase margins are often used to quantify robustness of a feedback system. However, even if the gain and phase margins are both large, a small perturbation in both gain and phase simultaneously may still render the system unstable, see, e.g., \cite[pp.~239-240]{zhou1996robust}. This can be understood by considering the forbidden area in the Nyquist plot for an independent gain and phase margin of $k$ and $\varphi$, respectively, as illustrated in Figure~\ref{subfig:nyquist_gain_phase_indep}. To remedy this, one can consider the problem of simultaneous gain and phase margin, i.e., perturbations $\Delta(s) \equiv \kappa e^{-i \theta}$ where $\kappa \in [1, k]$ and $\theta \in [-\varphi, \varphi]$. This gives forbidden areas for the Nyquist curve and complementary sensitivity functions illustrated in Figures~\ref{subfig:nyquist_gain_phase_margin} and \ref{subfig:interpolation_gain_phase}, respectively. Equations for the boundaries of the forbidden areas for the complementary sensitivity function are provided in Appendix~\ref{app:sim_gain_phase}. For a desired simultaneous gain and phase margin, the forbidden region for the complementary sensitivity function can be mapped to the complement of a disc using a numerically computed conformal map, see, e.g., \cite{nehari1952conformal, trefethen1980numerical, bjorstad1987conformal, howell1993numerical, marshall2007convergence}, and thus the existence of a controller can be tested numerically. An approximate approach to this is to extend the forbidden area in the Nyquist plot to a disc. This gives rise to the so-called {\em disk margin\/} \cite{blight1994practical, sepulchre1997constructive}, which is conservative as an estimate for simultaneous gain and phase margin.

\begin{figure*}[tb]
\begin{center}
\begin{subfigure}[t]{.49\columnwidth}
 \centering
 \includegraphics[width=\textwidth]{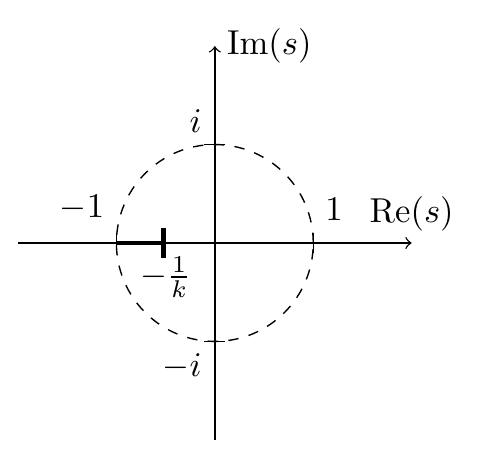}
 \subcaption{Gain margin.}
 \label{subfig:nyquist_gain_margin}
\end{subfigure}
\hfill
\begin{subfigure}[t]{.49\columnwidth}
\centering
  \includegraphics[width=\textwidth]{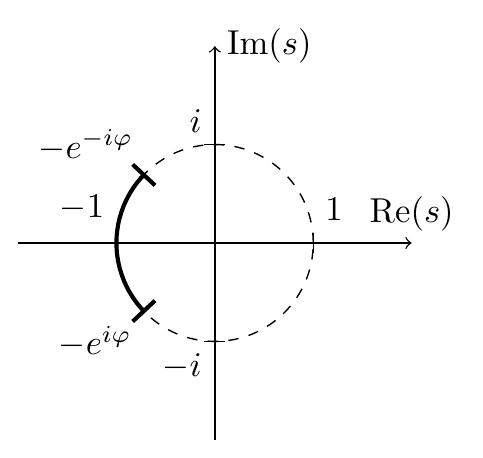}
 \subcaption{Phase margin.}
 \label{subfig:nyquist_phase_margin}
\end{subfigure}
\hfill
\begin{subfigure}[t]{.49\columnwidth}
 \centering
 \includegraphics[width=\textwidth]{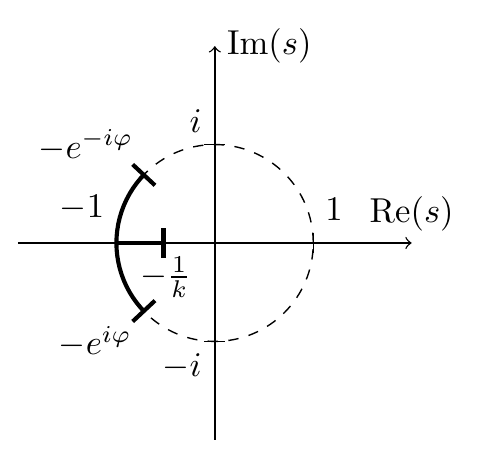}
 \subcaption{Independent gain and phase margins.}
 \label{subfig:nyquist_gain_phase_indep}
\end{subfigure}
\hfill
\begin{subfigure}[t]{.49\columnwidth}
 \centering
 \includegraphics[width=\textwidth]{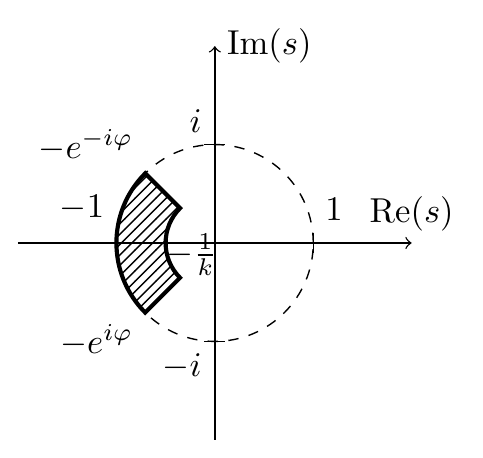}
 \subcaption{Simultaneous gain and phase margin.}
 \label{subfig:nyquist_gain_phase_margin}
\end{subfigure}
\caption{Illustration of forbidden areas in a Nyquist plot for gain maring, phase maring, independent gain and phase margin, and simultaneous gain and phase margin.}\label{fig:nyquist_gain_phase}
\end{center}
\end{figure*}

\begin{figure*}[tb]
\begin{center}
\begin{subfigure}[t]{.49\columnwidth}
 \centering
 \includegraphics[width=\textwidth]{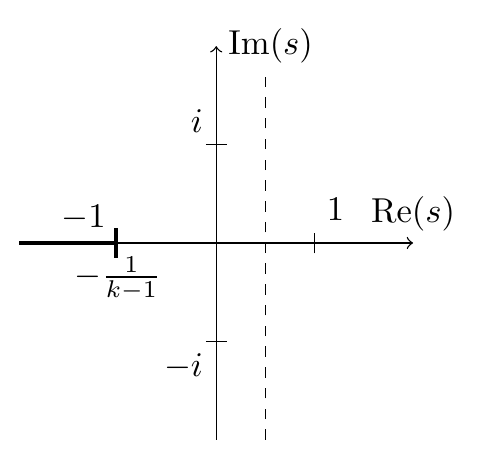}
 \subcaption{Gain margin.}
 \label{subfig:interpolation_gain}
\end{subfigure}
\hfill
\begin{subfigure}[t]{.49\columnwidth}
\centering
  \includegraphics[width=\textwidth]{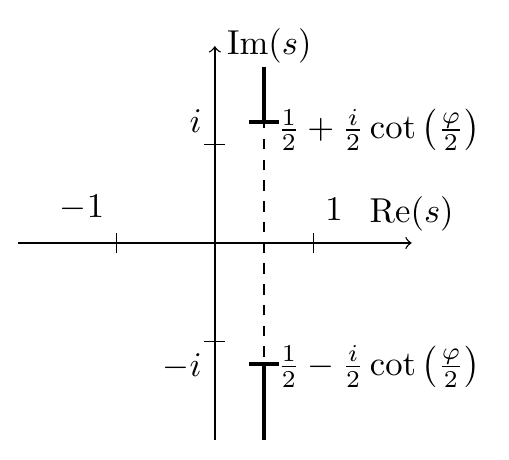}
 \subcaption{Phase margin.}
 \label{subfig:interpolation_phase}
\end{subfigure}
\hfill
\begin{subfigure}[t]{.49\columnwidth}
 \centering
 \includegraphics[width=\textwidth]{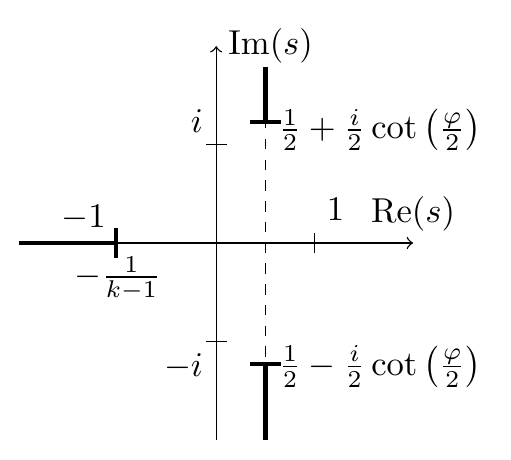}
 \subcaption{Independent gain and phase margins.}
 \label{subfig:interpolation_gain_phase_indep}
\end{subfigure}
\hfill
\begin{subfigure}[t]{.49\columnwidth}
 \centering
 \includegraphics[width=\textwidth]{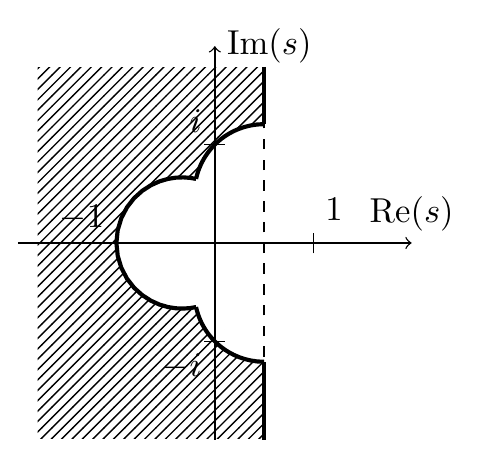}
 \subcaption{Simultaneous gain and phase margin.}
 \label{subfig:interpolation_gain_phase}
\end{subfigure}
\caption{
  Illustration of forbidden areas in the range of the interpolant for gain maring, phase maring, independent gain and phase margin, and simultaneous gain and phase margin. The dotted line, $0.5 + i\mR$, is the image of the unit cirlc under the M\"obius transformation \eqref{eq:mobius}.}\label{fig:interpolation_gain_phase}
\end{center}
\end{figure*}

\subsection{The delay margin problem}
\label{subsec:delay_margin}

In the delay margin problem the perturbations
%are of 
\rev{take} the form
\[
\Delta(s)=e^{-s\tau}, \quad \text{ for } \tau \in [0,\bar\tau],
\]
\rev{and the} \emph{delay margin} for a given controller $K$ is defined as
\begin{align*}
\mathscr{T}(P, K) & := \sup_{\tau \geq 0} \; \tau \\
  & \text{such that } K \text{ stabilizes } Pe^{-t s} \text{ for } t \in [0, \tau].
\end{align*}
The \emph{maximum delay margin} for a plant $P$ \rev{is defined} as
\begin{align}\label{eq:taumax_def}
\taumax(P)  & := \sup_{K \in \quotientfield} \; \mathscr{T}(P, K),
\end{align}
which is the largest value such that for any $\bar{\tau} < \taumax$ there exists a controller that stabilizes the plant $P$ for all $\tau\in [0, \bar{\tau}]$.

Given a controller $K$ that stabilizes the plant $P$, the smallest delay that destabilizes the system corresponds to the smallest $\tau$ such that $-e^{i\omega\tau}$ intersects the Nyquist curve $K(i\omega)P(i\omega)$. This means that the robust stabilization problem of systems with uncertain time delay can be understood as forbidden regions in the Nyquist plot, just like the problems with uncertain gain and phase. However, in the case of an uncertain delay, the uncertainty becomes a frequency dependent phase shift, and thus the forbidden region is a frequency dependent connected subset of the unit circle. In particular, for a frequency $\omega$ the forbidden area is give by the arc $\{ -e^{i\omega\tau} \mid \tau \in [0, \bar{\tau}] \}$, as illustrated in Figures~\ref{subfig:nyquist_delay_margin_small_omega} and \ref{subfig:nyquist_delay_margin_large_omega} for two frequencies $0 < \omega < 2\pi/\bar{\tau}$. From this we also see that for $|\omega| \geq 2\pi/\bar{\tau}$, the Nyquist curve must remain either inside or outside the unit circle. However, the latter corresponds to that the loop gain $|P(\infty)K(\infty)| > 1$, which is unfeasible from a control perspective since any system model inevitably contains modeling errors at sufficiently high frequencies \cite[p.~35]{doyle1992feedback}.

As in the gain and phase margin problems, the forbidden region in the Nyquist plot can be translated into forbidden regions for the complementary sensitivity function $T$ using the M\"obius transformation \eqref{eq:mobius}. Since the unit circle is mapped to the line with real part $1/2$, the forbidden region will be frequency dependent cuts, $\cut{\omega \bar{\tau}}$ as defined in \eqref{eq:def_cut}, that will become increasingly restrictive as the modulus of the frequency, $|\omega|$, grows. This is illustrated in Figures~\ref{subfig:interpolation_delay_margin_small_omega} and \ref{subfig:interpolation_delay_margin_large_omega}. This connection will be further investigated in Section~\ref{subsec:delay_nyquist}.

Note that if the plant $P$ is stable we trivially have $\taumax = \infty$, since $K \equiv 0$ stabilizes it. The same observation holds for all stability margins considered in this paper,  and thus we shall only consider unstable plants. In contrast with the maximum gain and phase margin problems, the maximum delay margin problem is  unsolved. However, work has been done to obtain lower and upper bounds.

\begin{figure}[tb]
\begin{center}
\begin{subfigure}{.49\columnwidth}
 \centering
 \includegraphics[width=\textwidth]{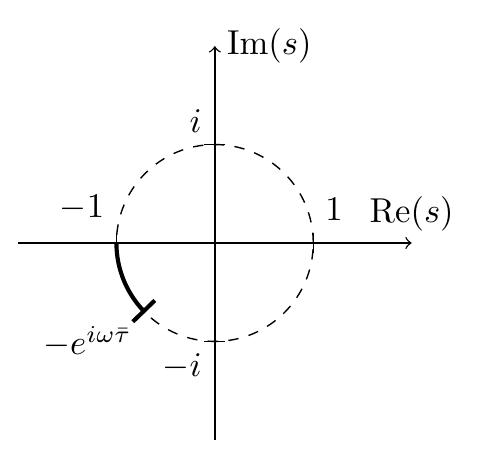}
 \subcaption{Forbidden areas in the Nyquist plot for frequency $\omega = \pi/(2\bar{\tau})$.}
 \label{subfig:nyquist_delay_margin_small_omega}
\end{subfigure}
\hfill
\begin{subfigure}{.49\columnwidth}
 \centering
 \includegraphics[width=\textwidth]{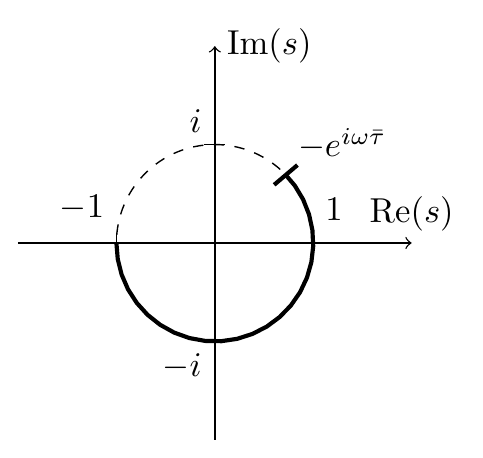}
 \subcaption{Forbidden areas in the Nyquist plot for frequency $\omega = 3\pi/(2\bar{\tau})$.}
 \label{subfig:nyquist_delay_margin_large_omega}
\end{subfigure}
\\
\begin{subfigure}{.49\columnwidth}
\centering
  \includegraphics[width=\textwidth]{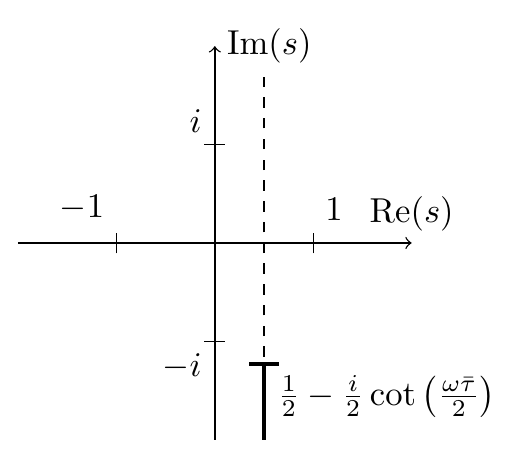}
 \subcaption{Forbidden areas in the range of the interpolant for frequency $\omega = \pi/(2\bar{\tau})$.}
 \label{subfig:interpolation_delay_margin_small_omega}
\end{subfigure}
\hfill
\begin{subfigure}{.49\columnwidth}
\centering
  \includegraphics[width=\textwidth]{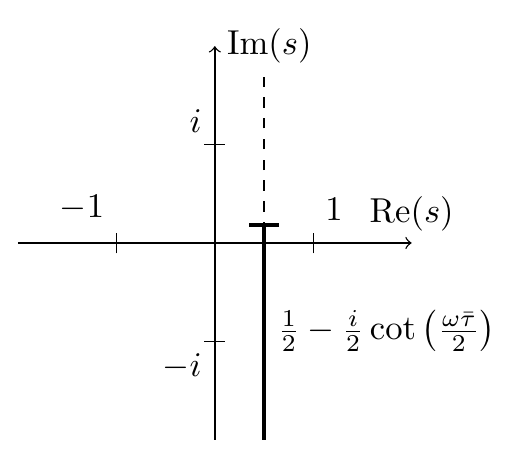}
 \subcaption{Forbidden areas in the range of the interpolant for frequency $\omega = 3\pi/(2\bar{\tau})$.}
 \label{subfig:interpolation_delay_margin_large_omega}
\end{subfigure}
\caption{Illustration of forbidden region, in both the Nyquist plot and for the range of the interpolant, for the delay margin problem. The regions are frequency dependent, and the areas are illustrated for two different frequencies.
}\label{fig:nyquist_interpolation_delay_margin}
\end{center}
\end{figure}

Upper bounds for the maximum delay margin problem have been studied in \cite{middleton2007achievable}.  In fact, the results in \cite{middleton2007achievable} are the first that show that there is an upper bound of the achievable delay margin when using LTI controllers, and a region for the delay where stabilization is not possible is described. However, these bounds are in general not tight, except for some special cases,
\rev{for example the case of real plants with one unstable pole $p$ and a potential nonminimum phase zero $z$ with $p < z$.
For other cases, these bounds}
have recently also been improved upon in \cite{ju2016further, ju2018achievable}.

To ensure stability we are in general more interested in a lower bound $\bar{\tau}\leq \taumax$, and this is also the problem considered in this paper. As starting point we take  the recent approach from \cite{qi2014fundamental,qi2017fundamental}, where a method for computing a lower bound was constructed based on analytic interpolation and rational approximation.
To this end, we observe that the condition $1 + P(s)K(s)e^{-\tau s} \neq 0$ for all $s \in \bar \mC_+$ and all $\tau \in [0, \bar{\tau}]$, corresponding to condition \eqref{eq:stability} for robust stabilization of a delay system with delay uncertainty $\tau \in [0, \bar{\tau}]$, can be written in terms of the the complementary sensitivity function as
\begin{equation}\label{eq:nec_and_suff_cond}
1+ T(s)(e^{-\tau s}-1)\neq 0 \quad \text{for  $s\in \bar{\mC}_+$ and $\tau \in [0, \bar{\tau}]$}.
\end{equation}
A sufficient condition for \eqref{eq:nec_and_suff_cond} to hold is that there exists a $T$ satisfying the interpolation conditions \eqref{eq:interpolation} such that
\begin{equation} \label{eq:suff_cond0}
\sup_{\tau\in [0, \bar{\tau}]} \|T(s)(e^{-\tau s}-1)\|_{\Hinf}<1.  
\end{equation}
By noting that the $\Hinf$-norm is defined by a supremum, and since $T \in \Hinf$, chaining the order of the two suprema leads to the equivalent condition 
\[
\|T(i \omega) \phi_{\bar{\tau}}(\omega)  \|_{\Linf}<1,
\]
where
\begin{align}
\phi_{\bar{\tau}}(\omega) & =  \sup_{\tau\in [0, \bar{\tau}]}  | e^{-i\tau \omega}-1| \nonumber \\
& = \begin{cases}2\left|\sin(\frac{\bar{\tau} \omega}{2})\right| &\mbox{ for } |\omega \bar{\tau}| \le \pi\\
2 &\mbox{ for } |\omega \bar{\tau}| > \pi.
\end{cases} \label{eq:w_tau}
\end{align}
Therefore, a sufficient condition for the existence of a stabilizing controller with delay margin $\bar{\tau}$ is that
\begin{equation}\label{eq:suff_cond}
\inf_{\substack{T \in \Hinf \\ \text{subject to } \eqref{eq:interpolation}}} \|T(i \omega) \phi_{\bar{\tau}}(\omega)  \|_{\Linf}<1.
\end{equation}

However, \eqref{eq:w_tau} is an infinite-dimensional function, i.e., not rational. To overcome this, in \cite{qi2017fundamental} the function $\phi_{\bar{\tau}}$ is approximated by the magnitude of a stable, minimum-phase rational function $w_{\bar{\tau}}$ such that $\phi_{\bar{\tau}}(\omega)\leq |w_{\bar{\tau}}(i\omega)|$ for all $\omega \in \mR$. Using this approximation and the interpolation conditions on $T$ for internal stability, the authors derive an algorithm for computing the largest $\bar{\tau}$ for which \eqref{eq:suff_cond} holds. This thus gives a lower bound for the maximum delay margin.

\section{Lower bounds on the delay margin via analytic interpolation}
\label{sec:lower_bound_basic}

To introduce our analytic interpolation setting we first derive an algorithm previously presented in \cite{ringh2018lower} as a modification (and possible improvment) of that in \cite{qi2014fundamental, qi2017fundamental}. To this end, note that \eqref{eq:suff_cond}, i.e., the sufficient condition for the closed loop system to be internally stable for all $\tau\in [0, \bar{\tau}]$, holds if there exists a $T(s) \in \HinfC$ such that
\begin{equation}\label{eq:basic_interpolation_problem_1}
\| T(i\omega) \phi_{\bar{\tau}}(\omega) \|_{\Linf} \! <  1
\; \text{and}
\begin{cases} T(p_j)=1, \; j = 1, \ldots, n, \\ T(z_j)=0, \; j = 1, \ldots, m.
\end{cases} \!\!\!\!\!\!\!
\end{equation}
Next, observe that $\int_{-\infty}^{\infty}\log(\phi_{\bar{\tau}}(\omega))/(1 + \omega^2) d\omega < \infty$. This means that $\phi_{\bar{\tau}}$ can be extended from a function on the imaginary axis to a so-called outer function $W_{\bar{\tau}} \in \HinfC$ \cite[pp.~132-133]{hoffman1962banach}.%
\footnote{In the control literature, an outer function is often referred to as a minimum phase function (cf. \cite[pp.~93-94]{doyle1992feedback}).}
The extension, which is such that $W_{\bar{\tau}}$ has the same magnitude as $\phi_{\bar{\tau}}$ almost everywhere on $i\mR$, is given by   \cite[p.~133]{hoffman1962banach}
\begin{equation}
\label{eq:outerrepr}
W_{\bar{\tau}}(s) = \exp \! \left[ \frac{1}{\pi} \int_{-\infty}^\infty \!\! \log \big( \phi_{\bar{\tau}}(\omega) \big) \frac{ \omega s + i}{\omega + is} \frac{1}{1+\omega^2} \,d\omega \right] \! . \!
\end{equation}
This means that \eqref{eq:basic_interpolation_problem_1} is equivalent to
\begin{equation}\label{eq:basic_interpolation_problem_15}
\| T W_{\bar{\tau}} \|_{\Hinf} < 1
\; \text{and}
\begin{cases} T(p_j)=1, \; j = 1, \ldots, n, \\ T(z_j)=0, \; j = 1, \ldots, m.
\end{cases}
\end{equation}
Now, using the fact that $W_{\bar{\tau}}$ is outer and hence that it has no poles or zeros in $\mC_+$, we set $\tilde{T}:=TW_{\bar{\tau}}$ and note that if $\tilde{T} \in \Hinf$, then the corresponding complementary sensitivity function $T = \tilde{T}W_{\bar{\tau}}^{-1}$ is also analytic in $\mC_+$ and satisfies the interpolation conditions.%
\footnote{Note that $W_{\bar \tau}(0)=0$, and hence $T$ has a pole in $0$. To overcome this problem, one can impose a lower bound on $\phi_{\bar \tau}$ which ensures that $T\in \HinfC$, see Section~\ref{sec:control_implementation}.}
 An equivalent problem to \eqref{eq:basic_interpolation_problem_15}
is therefore 
\begin{equation}\label{eq:basic_interpolation_problem_2}
\| \tilde T \|_{\Hinf} \! < 1
\; \text{and}
\begin{cases}
 \tilde T(p_j)=W_{\bar{\tau}}(p_j), &j = 1, \ldots, n, \\
 \tilde T(z_j)=0,  &j = 1, \ldots, m,
 \end{cases}
\end{equation}
and thus the only way the weight enters is through the values of the outer function $W_{\bar{\tau}}$ at the pole locations $p_j$ \cite[Sec.~III.B]{zames1983feedback},  \cite{kimura1984robust}, \cite[Sec.~4.C]{karlsson2010theinverse} (cf. \cite[Prop.~7]{karlsson2009degree}).

Thus \eqref{eq:suff_cond} holds if and only if
there is a $\tilde{T} \in \HinfC$ such that \eqref{eq:basic_interpolation_problem_2} holds.
The values $W_{\bar{\tau}}(p_j)$,  $ j = 1, \ldots, n$, can be computed from \eqref{eq:outerrepr} by numerical integration. Moreover, setting 
\begin{subequations}\label{eq:pick_values}
\begin{align}
    &v := [p_1, \ldots, p_n, z_1, \ldots, z_m]   \\
    & w := [W_{\bar{\tau}}(p_1), \ldots, W_{\bar{\tau}}(p_n), 0, \ldots, 0] ,
\end{align}
\end{subequations}
 the interpolation problem \eqref{eq:basic_interpolation_problem_2} is solvable if and only if the corresponding Pick matrix
\begin{equation}\label{eq:pick_matrix}
\text{Pick}(v,w) := \begin{bmatrix}
\frac{1-w_j\bar w_k}{v_j+\bar{v}_k}
\end{bmatrix}_{j,k = 1}^{n+m}
\end{equation}
is positive definite;  see, e.g.,
\cite[pp.~157-159]{doyle1992feedback}.
This is true for distinct poles and zeros. If this is not the case, \eqref{eq:pick_matrix} needs to be replaced by a more general criterion, e.g., using the input-to-state framework \cite{byrnes2001ageneralized,georgiou2002structure}
as in \cite{blomqvist2005optimization}.

To summarize, this means that for a given $\bar \tau$, the problem \eqref{eq:suff_cond} has a solution if and only if the Pick matrix \eqref{eq:pick_matrix} with interpolation values \eqref{eq:pick_values} is positive definite. Moreover, since $\phi_{\bar{\tau}}(\omega)$ is point-wise nondecreasing in $\bar{\tau}$, if \eqref{eq:suff_cond} has a solution for some $\bar{\tau}$ then it has a solution for any smaller value of $\bar{\tau}$. The latter means that the optimal $\bar \tau$ can be computed using bisection, iteratively testing the feasibility of \eqref{eq:suff_cond} by evaluating the Pick matrix \eqref{eq:pick_matrix}. The method is summarized in Algorithm~\ref{alg:first_method}. Note that by using \cite[Thms.~7, 9, and 10]{middleton2007achievable} we have that $2\pi/\max_j(|p_j|)\ge\tau_{\rm max}$, which gives a valid choice for the initial upper bound in the bisection algorithm.

\renewcommand{\algorithmicrequire}{\textbf{Input:}}
\renewcommand{\algorithmicensure}{\textbf{Output:}}
\algsetup{indent=12pt}

\begin{algorithm}[tb]
\caption{Lower bound on maximum delay margin}
\label{alg:first_method}
\begin{algorithmic}[1]
\REQUIRE Unstable poles $p_j$,  $j = 1, \ldots, n$, and nonminimum phase zeros $z_j$,  $j = 1, \ldots, m$, of the plant $P$.

\STATE $\tau_- = 0.$%, $\quad \tau_{\rm mid} = (\tau_+ + \tau_-)/2$
\STATE $\tau_+=2\pi/\max_j(|p_j|)$, %$\quad 
\WHILE{$\tau_+ - \tau_- > \texttt{tol}$}
\STATE $\tau_{\rm mid} = (\tau_+ + \tau_-)/2$
\STATE Compute new interpolation values $W_{\tau_{\rm mid}}(p_j)$
\IF{Pick matrix \eqref{eq:pick_matrix} with values \eqref{eq:pick_values} is positive definite}
\STATE $\tau_- = \tau_{\rm mid}$
\ELSE
\STATE $\tau_+ = \tau_{\rm mid}$
\ENDIF
\ENDWHILE
\STATE $\bar{\tau}
 = \tau_-$
\ENSURE $\bar{\tau}
$, lower bound on maximum delay margin
\end{algorithmic}
\end{algorithm}

The improvement of this method over that in \cite{qi2017fundamental} depends on how well the magnitude of the fifth-order approximation $w_{6\tau}(i\omega)$ used in \cite{qi2017fundamental} fits $\phi_{\bar{\tau}}(\omega)$ for $\omega\in \mR$. However, the the formulation presented here allows for interpretations and extensions to be presented in the coming sections.

\subsection{Interpretation in terms of the Nyquist plot}
\label{subsec:delay_nyquist}

As described in Section~\ref{subsec:delay_margin}, in order to stabilize the system for all delays $\tau \in [0, \bar{\tau}]$ the complementary sensitivity function must fulfill the interpolation conditions \eqref{eq:interpolation} and satisfy \eqref{eq:nec_and_suff_cond}. Moreover, in that section it was also argued, using constraints on the Nyquist curve and the M\"obius transformation \eqref{eq:mobius}, that robust stability could also be characterized by $T(i\omega)$ not intersecting $\cut{\omega \bar{\tau}}$. For a set $A$ define
\[
\dist(A,x) := \inf_{y \in A} | x - y |
\]
to be the distance from a point $x$ to $A$. Then the observations in Section~\ref{subsec:delay_margin} can be formalized as follows.

\begin{thm}\label{thm:dist_to_cut}
Let $T \in \HinfC$, and assume that $T(\infty)$ is well-defined and that $\real(T(\infty)) < 1/2$. Moreover, let $\bar{\tau} > 0$. Then the following two statements are equivalent:
\begin{enumerate}
\item there exists an $\varepsilon > 0$ such that $|1 + T(s)(e^{-\tau s} + 1)| \geq \varepsilon$ for all $s \in \bar{\mC}_+$ and $\tau \in [0, \bar{\tau}]$
\item there exists an $\varepsilon > 0$ such that $\dist ( \cut{\omega\bar{\tau}}, T(i\omega)) \geq \varepsilon$ for all $\omega \in \mR$.
\end{enumerate}
\end{thm}

\begin{proof}
See Appendix~\ref{app:dist_to_cut}.
\end{proof}

\begin{remark}\label{rem:T_inf_less_than_half}
The condition $\real(T(\infty)) < 1/2$ can be understood in terms of the loop gain. To see this, note that $\real(T(\infty)) \geq 1/2$ corresponds to that the loop gain $|P(\infty)K(\infty)| \geq 1$. As mentioned in Section~\ref{subsec:delay_margin} this is undesirable for control systems, see \cite[p.~35]{doyle1992feedback}. Also note that if $\real(T(\infty)) \geq 1/2$, then there is a (unbounded) sequence $(s_n)_n$ such that $1 + T(s_n)(e^{-\tau s_n} + 1)  \to 0$ for any $\tau > 0$; see the proof of Lemma~\ref{lem:T_0_infty} in Appendix~\ref{app:proof_lem_T_0_infty}.
\end{remark}

The frequency dependent forbidden region for $T$ provides an interpretation of the weight function $\phi_{\bar{\tau}}$. To see this, consider the constraint $\| T(i\omega) \phi_{\bar{\tau}}(\omega) \|_{\Linf} < 1$ in \eqref{eq:basic_interpolation_problem_1}. For each frequency $\omega$, this confines the complementary sensitivity function, $T(i\omega)$, to a ball centered at the origin and with radius $| \phi_{\bar{\tau}}(\omega)|^{-1}$. By a direct calculation, it can be verified that this corresponds to the minimal distance between the origin and the set $\cut{\omega \bar{\tau}}$.

Finally, for an unstable plant $P$, the Nyquist stability criterion \cite[p.~39]{doyle1992feedback} requires that the Nyquist curve encircles the point $-1$ at least once. For a sufficiently large frequency, $|\omega| > 2\pi/\bar{\tau}$, the forbidden arc will fill the whole unit circle. So in order to stabilize for all delays $\tau \in [0, \bar{\tau}]$, the encirclement must happen at a sufficiently low frequency. Analogously, the complementary sensitivity function takes the value $1$ at the unstable poles, thus the curve $T(i\omega)$ must have real part larger than $1$ for some frequencies. This amounts to passing back and forth between cuts before the gap closes. To facilitate this, we will next introduce a frequency dependent shift.

\section{Improving the lower bound using a shift}
\label{sec:improved_method}

As noted above, the constraint $| T(i\omega) \phi_{\bar{\tau}}(\omega) | < 1$ forces $T$ to take values in a disc centered at the origin, where the radius is given by the distance between the origin and the cut $\cut{\omega \bar{\tau}}$. However, choosing the center of the disc at the origin is quite arbitrary, and by instead carefully selecting the center elsewhere, we may improve the estimate of the lower bound. To this end, consider
\begin{equation}\label{eq:T0}
T(s) = \hat{T}(s) + T_0(s),
\end{equation}
where $T_0(i\omega)$ represents the center of the disc at frequency $\omega$, and $\hat{T}(i\omega)$ will be constrained so that $T(i\omega)$ does not intersect the cut. Since we require $T \in \HinfC$, we need $T_0$ and $\hat{T}$ to also be in $\HinfC$.

With this construction, and by using \eqref{eq:nec_and_suff_cond}, the feedback system is internally stable for all $\tau \in [0, \bar{\tau}]$ if $T \in \HinfC$ fulfills the interpolation conditions \eqref{eq:interpolation}, and  we have that
\begin{equation}
\label{eq:nec_and_suff_cond_modified}
\hat{T}(s) \big( e^{-\tau s} - 1 \big) \neq -1 + T_0(s)(1 - e^{-\tau s})
\end{equation}
for all $s \in \bar{\mC}_+$ and all $\tau \in [0, \bar{\tau}]$. By Theorem~\ref{thm:dist_to_cut}, the right hand side of \eqref{eq:nec_and_suff_cond_modified} is nonzero in $\bar{\mC}_+$ if $T_0(i\omega)$ does not intersect $\cut{\omega \bar{\tau}}$. Note, however, that $T_0$ is not required to satisfy the interpolation conditions \eqref{eq:interpolation}. Since the right hand side does not have any zeros in $\bar{\mC}_+$  the inverse is also in $\HinfC$ and thus \eqref{eq:nec_and_suff_cond_modified} can be written as
\begin{equation}\label{eq:new_nec_and_suff_cond}
\hat{T}(s) \frac{e^{-\tau s} - 1}{1 - T_0(s)(1 - e^{-\tau s})} \neq -1.
\end{equation}
Therefore we need to modify the function $\absfuncwzero$ in 
\rev{\eqref{eq:w_tau}}
to read
\begin{equation}\label{eq:weight_with_T0}
\absfuncwzero(\omega) := \sup_{\tau \in [0, \bar{\tau}
]} \left| \frac{e^{-\tau i\omega} - 1}{1 - T_0(i\omega)(1 - e^{-\tau i\omega})} \right|.
\end{equation}
As shown in Appendix~\ref{app:phi_inv}, this implies that
\begin{equation}\label{eq:phi_T_zero}
\absfuncwzero(\omega)^{-1} = \dist(\cut{\omega\bar{\tau}}, T_0(i\omega)),
\end{equation}
which is precisely what we wanted to achieve by introducing $T_0$ in \eqref{eq:T0}.
Moreover, this distance can be computed more explicitly, and the expression is given in Appendix~\ref{app:phi_inv}. Furthermore, note that \eqref{eq:new_nec_and_suff_cond} reduces to \eqref{eq:nec_and_suff_cond} when $T_0(s) \equiv 0$. Using the same argument as before, we see that 
\begin{equation}\label{eq:Thatt_ball}
\| \hat{T}(i \omega) \absfuncwzero(\omega) \|_{\Linf}<1
\end{equation}
is  a sufficient condition for  \eqref{eq:new_nec_and_suff_cond} to hold. 
In the special case that $T_0(s) \equiv T_0$ is constant, which was considered in \cite{ringh2018lower},  $-1 + T_0(1 - e^{-\tau s})$ is nonzero in $\bar{\mC}_+$ if \rev{and only if} $\real(T_0) < 1/2$. Similarly, in the general case, we have the following condition.

\begin{lemma}\label{lem:T_0_infty}
Assume that $T_0(\infty)$ is well-defined. If $\real(T_0(\infty)) < 1/2$ then there exists a $\bar{\tau} > 0$ such that $(1 - T_0(s)(1 - e^{-s\tau}))^{-1} \in \HinfC$ for all $\tau \in [0, \bar{\tau}]$. Conversely, if $\real(T_0(\infty)) \geq 1/2$, then $(1 - T_0(s)(1 - e^{-s\tau}))^{-1}$ is not bounded in $\bar{\mC}_+$ for any $\tau > 0$.
\end{lemma}

\begin{proof}
See Appendix~\ref{app:proof_lem_T_0_infty}.
\end{proof}

Note that the function $\absfuncwzero(\omega)$ is log integrable for sufficiently small $\bar\tau>0$. Thus we can follow the same procedure as in Section~\ref{sec:lower_bound_basic} and define, via the representation \eqref{eq:outerrepr}, an outer function $\outerfuncwzero(s)$ with the property $|\outerfuncwzero(i\omega)|=\absfuncwzero(\omega)$ for almost all $\omega\in \mR$. Consequently, we are left with the problem to find a $\hat T$ such that
\begin{equation*}
\|\hat T \outerfuncwzero \|_{\Hinf}<1
\; \text{and} \;
\begin{cases}
\hat T(p_j)=1-T_0(p_j),  &j = 1, \ldots, n, \\
\hat T(z_j)=-T_0(z_j),  &j = 1, \ldots, m,
\end{cases}
\end{equation*}
which, in turn,  is equivalent to 
\begin{subequations}\label{eq:new_interpolation}
\begin{align}
& \|\tilde T\|_{\Hinf} < 1, \text{ and}  \label{eq:new_interpolation_a}\\
& \tilde T(p_j)=(1-T_0(p_j))\outerfuncwzero(p_j),  &\!\! j = 1, \ldots, n, \label{eq:new_interpolation_b}\\
& \tilde T(z_j)=-T_0(z_j) \outerfuncwzero(z_j),  &\!\! j = 1, \ldots, m.\label{eq:new_interpolation_c}
\end{align}
\end{subequations}
In the same manner as in Section~\ref{sec:lower_bound_basic} we can then determine feasibility by checking whether the corresponding Pick matrix \eqref{eq:pick_matrix} is positive definite. A refined algorithm for computing a lower bound for the maximum delay margin is thus obtained by suitable changes in Algorithm~\ref{alg:first_method}.

\begin{remark}\label{rem:T01}
Note that if $T_0$ is selected so that it satisfies the interpolation conditions \eqref{eq:interpolation}, then the above construction is valid and trivially satisfied by $\tilde T\equiv 0$ for every $\bar\tau$ small enough so that $T_0(i\omega)$ does not intersect the cut. Hence the supremum of such $\bar\tau$ is the delay margin of the controller $K_0$ corresponding to $T_0$. This observation amounts to the procedure in \cite[Rem.~1]{middleton2007achievable} for computing $\mathscr{T}(P, K_0)$. This can be seen by using the transformation $s/(1-s)$, i.e., the inverse of \eqref{eq:mobius}, which brings the complementary sensitivity function $T_0$ back to the loop gain $PK_0$ in the Nyquist plot.
\end{remark}

\begin{remark}\label{rem:T02}
Let $\real(T_0(i\omega_0)) = 1/2$, for some $\omega_0 > 0$. Then an upper bound for which the approach is valid is given by $\bar{\tau} \leq 2\cot^{-1}(-2 \cdot \imag(T_0(i\omega_0)) / \omega_0$. An equivalent statement is true also for negative frequencies. Moreover, this means that any choice of $T_0$ such that $\real(T_0(i\omega)) = 1/2$ for some $\omega$ will put an upper limit on the lower bound of the maximum delay margin that a method based on solvability of \eqref{eq:new_interpolation} can achieve. Thus, when designing $T_0$ one should make sure that it does not cross the line $0.5+i\mR$ for frequencies with large module. 
\end{remark}

To conclude this section we note that by appropriately selecting $T_0$ the method proposed here, based on solvability of \eqref{eq:new_interpolation}, can achieve a lower bound on the maximum delay margin that is arbitrarily close to the true value.

\begin{prop}
For any $\bar{\tau} \in (0, \taumax)$ there exists a function $T_0(z) \in \HinfC$ such that there is a $\tilde T \in \HinfC$ which satisfies \eqref{eq:new_interpolation} for this $\bar{\tau}$.
\end{prop}

\begin{proof}
Consider a sequence of controllers $(K^{(n)})_n$ such that  $\taumax(P) - \mathscr{T}(P, K^{(n)}) \leq c/n$, for some constant $c > 0$, i.e., a sequence that converges to the supremum in \eqref{eq:taumax_def}. Such a sequence always exists, and from this sequence we form the sequence of corresponding complementary sensitivity functions $(T^{(n)})_n$, i.e., $T^{(n)} := PK^{(n)}/(1 + PK^{(n)})$. Taking $T_0 = T^{(n)}$, by using the observation in Remark~\ref{rem:T01} we conclude that a feasible solution to \eqref{eq:new_interpolation} is $\tilde{T} \equiv 0$. Since this is true for all values of $n$, the conclusion follows.
\end{proof}

\subsection{\rev{A systems interpretation of $T_0$}}
\label{subsec:small_gain_interp}
\rev{The nominal complementary sensitivity function $T_0$ was introduced as a the center of the ball to which the analytic interpolant $\hat{T}$ is confined, as shown \eqref{eq:phi_T_zero} and \eqref{eq:Thatt_ball}. However, $T_0$ also has a systems theoretic interpretation. To this end, we first examine the case when $T_0(s) \equiv 0$, corresponding to the method derived in Section~\ref{sec:lower_bound_basic}. Now, consider the feedback interconnection between the systems $\Delta(s)$ and $T(s)$ as in Figure~\ref{fig:robust_blockdiagram}. By the small gain theorem, this feedback interconnection is internally stable if each system is internally stable and if $|\Delta(i\omega) T(i\omega) | < 1$ for all $\omega \in \mR$ \cite[Sec.~1.6]{glad2000control}. }% 
\rev{With this in mind, }%
the method derived in Section~\ref{sec:lower_bound_basic}, which builds on the sufficient condition
\rev{\eqref{eq:nec_and_suff_cond}}, can be viewed as an application of the small gain theorem to $T(s)$ and the family of functions $\Delta(s) \in  \Omega:=\{1-e^{-\tau s} \mid \tau\in [0,\bar\tau]\}$. This\rev{, in turn,} is equivalent to applying the small gain theorem to $T(s)$ and the function $W_{\bar{\tau}}(s)$ in \eqref{eq:outerrepr}. This equivalence holds since, by definition, the magnitude $|W_{\bar{\tau}}(i\omega)|$ is the maximum magnitude  $|\Delta(i\omega)|$ for all $\Delta \in \Omega$ and $\omega\in\mR$.

A similar interpretation can be made for the method in Section~\ref{sec:improved_method}, \rev{which then also gives a systems theoretic interpretation of $T_0$. To do so, consider the interconnection in Figure~\ref{fig:robust_blockdiagram_T0}. In particular, note that as long as all signals in Figure~\ref{fig:robust_blockdiagram_T0} are bounded, then the signals in the points marked $a$ and $b$ are the same in both figures.
Therefore, by instead considering the two system $\hat\Delta(s)$ and $\hat T(s)$ we have that if both these systems are internally stable, then stability in the interconnection in Figure~\ref{fig:robust_blockdiagram_T0} implies that the interconnection in Figure~\ref{fig:robust_blockdiagram} is stable.
Moreover, this means that the family $\Omega$ of disturbances is replaced by $\hat\Omega:=\{\Delta(1- T_0\Delta)^{-1} \mid \Delta\in\Omega\}$, where we see that $T_0$ can be used to (partly) shape the disturbances, as long as the entire set is still internally stable.
}%
In view of Theorem~\ref{thm:dist_to_cut}, \rev{the latter}
means that the cuts corresponding to delay $\bar \tau$ must not intersect $T_0(i\omega)$.
This is equivalent to the right hand side of \eqref{eq:nec_and_suff_cond_modified} being invertible in $\HinfC$ for all $\tau\in[0,\bar \tau]$. This is also the reason why the method can not be applied for large delays $\bar \tau$ when $T_0$ intersects the line $0.5+i\mR$, as explained in Remark~\ref{rem:T02}.
\rev{Finally, the method in Section~\ref{sec:improved_method} can therefore be interpreted as applying the small gain theorem to $\hat T(s)$ and the family of systems $\hat\Omega$.
}

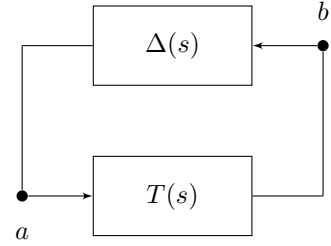
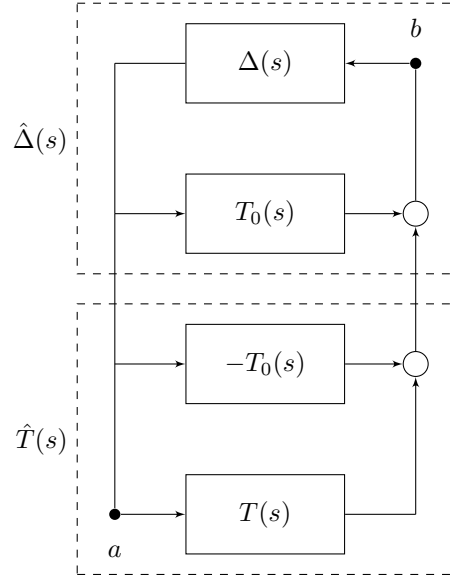
\begin{figure}[tbh]
\begin{center}
\begin{subfigure}{\columnwidth}
  
  \tikzstyle{int}=[draw, minimum size=2em]
  \tikzstyle{init} = [pin edge={to-,thin,black}]
  \tikzstyle{block} = [draw, rectangle, 
    minimum height=3em, minimum width=6em]
  \tikzstyle{sum} = [draw, circle, node distance=1cm]
  \tikzstyle{input} = [coordinate]
  \tikzstyle{output} = [coordinate]
  \tikzstyle{pinstyle} = [pin edge={to-,thin,black}]
  \tikzstyle{dott} =[circle,fill,inner sep=1.5pt]
  
  \begin{center}  
  \begin{tikzpicture}[auto, node distance=2cm,>=latex']
    % We start by placing the blocks
    \node [input, name=upperleft] {};
    \node [block, right of=upperleft] (disturbance) {$\Delta(s)$};
    \node [dott, right of=disturbance] (upperright) {};
    \node [input, below of=upperright] (lowerright) {};
    \node [block, left of=lowerright] (system) {$T(s)$};
    \node [dott, below of=upperleft] (lowerleft) {};

    \node [above=4pt of upperright] (pointb) {$b$};    
    \node [below=6pt of lowerleft] (pointa) {$a$};

    %Connect the blocks. 
    \draw [<-] (disturbance) -- (upperright);
    \draw [-] (upperright) -- (lowerright);
    \draw [-] (lowerright) -- (system);
    \draw [<-] (system) -- (lowerleft);
    \draw [-] (lowerleft) -- (upperleft);
    \draw [-] (upperleft) -- (disturbance);

  \end{tikzpicture}
  \subcaption{Block diagram interpretation of the delay uncertainty as presented in Section~\ref{sec:lower_bound_basic}. Here, $\Delta(s) \in  \Omega:=\{1-e^{-\tau s} \mid \tau\in [0,\bar\tau]\}$.}
  \label{fig:robust_blockdiagram}
  \end{center}
  \vspace{12pt}
\end{subfigure}
\begin{subfigure}{\columnwidth}

  \tikzstyle{int}=[draw, minimum size=2em]
  \tikzstyle{init} = [pin edge={to-,thin,black}]
  \tikzstyle{block} = [draw, rectangle, 
    minimum height=3em, minimum width=6em]
  \tikzstyle{sum} = [draw, circle]
  \tikzstyle{input} = [coordinate]
  \tikzstyle{output} = [coordinate]
  \tikzstyle{pinstyle} = [pin edge={to-,thin,black}]
  \tikzstyle{dott} =[circle,fill,inner sep=1.5pt]

  \begin{center}
  \begin{tikzpicture}[auto, node distance=2cm,>=latex']
    % We start by placing the blocks
    \node [input, name=upperleft] {};
    \node [block, right of=upperleft] (disturbance) {$\Delta(s)$};
    
    \node [dott, right of=disturbance] (upperright) {};
    \node [above=4pt of upperright] (pointb) {$b$};    
    
    \node [sum, below of=upperright] (uppermiddleright) {};
    \node [sum, below of=uppermiddleright] (lowermiddleright) {};
    \node [input, below of=lowermiddleright] (lowerright) {};
    \node [block, left of=lowerright] (system) {$T(s)$};
    
    \node [block, below of=disturbance] (T0) {$T_0(s)$};
    \node [block, below of=T0] (mT0) {$-T_0(s)$};
    
    \node [dott, left of=system] (lowerleft) {};
    \node [below=6pt of lowerleft] (pointa) {$a$};

    \coordinate [below of=upperleft] (uppermiddleleft);
    \coordinate [below of=uppermiddleleft] (lowermiddleleft);

    %Connect the blocks. 
    \draw [draw,->] (uppermiddleleft) -- (T0);
    \draw [draw,->] (T0) -- (uppermiddleright);
    \draw [draw,->] (lowermiddleleft) -- (mT0);
    \draw [draw,->] (mT0) -- (lowermiddleright);
    
    \draw [<-] (disturbance) -- (upperright);
    \draw [-] (upperright) -- (uppermiddleright);
    \draw [<-] (uppermiddleright) -- (lowermiddleright);
    \draw [<-] (lowermiddleright) -- (lowerright);
    \draw [-] (lowerright) -- (system);
    \draw [<-] (system) -- (lowerleft);
    \draw [-] (lowerleft) -- (lowermiddleleft);
    \draw [-] (lowermiddleleft) -- (uppermiddleleft);
    \draw [-] (uppermiddleleft) -- (upperleft);
    \draw [-] (upperleft) -- (disturbance);

    \draw [dashed] (-0.5,0.8) -- (4.5,0.8) -- (4.5,-2.8) -- (-0.5,-2.8) -- node {$\hat{\Delta}(s)$} (-0.5,0.8);
    \draw [dashed] (-0.5,-6.8) -- (4.5,-6.8) -- (4.5,-3.2) -- (-0.5,-3.2) -- node[left] {$\hat{T}(s)$} (-0.5,-6.8);

  \end{tikzpicture}
  \subcaption{Block diagram interpretation of the shift $T_0$. Here, $\hat{\Delta}(s) \in \hat\Omega:=\{\Delta(1- T_0\Delta)^{-1} \mid \Delta\in\Omega\}$, where $\Omega$ is as in Figure~\ref{fig:robust_blockdiagram}.}
  \label{fig:robust_blockdiagram_T0}
  \end{center}
\end{subfigure}
\caption{Block diagram representations of the delay uncertainties in Sections~\ref{sec:lower_bound_basic} and \ref{sec:improved_method}.}
\label{fig:robust_blockdiagram_both}
\end{center}
\end{figure}

\section{Considerations for controller design}
\label{sec:control_implementation} 

There are certain problems with the controller implementation that need attention. In particular, we want to guarantee that the complementary sensitivity function satisfies $T\in \HinfC$ and corresponds to a finite-dimensional controller.

\subsection{Solving the problem of $T$ having a pole in $s=0$}\label{subsec:pole_in_zero}

The approaches in Sections \ref{sec:lower_bound_basic} and \ref{sec:improved_method} gives a complementary  sensitivity function given by
\begin{equation}
\label{eq:Ttilde2T}
T(s)=\tilde{T}(s) \outerfuncwzero (s)^{-1}+T_0(s).
\end{equation}
The function $\outerfuncwzero$ can be inverted in $\mC_+$ since it is outer.  However, since $\outerfuncwzero(0)=0$, $T$ has a pole in $s = 0$, and therefore the closed loop system is not stable (cf. \cite[p.~37]{doyle1992feedback}). This can be rectified by replacing $\phi_{\bar{\tau}}$ by
\begin{equation}\label{eq:phi_lowerbound}
\phi_{\bar{\tau}, {\varepsilon}}(\omega) =\max(\varepsilon, \phi_{\bar{\tau}}(\omega))
\end{equation}
for some $\varepsilon > 0$, when defining $\outerfuncwzero$ in \eqref{eq:outerrepr}. This results in a stable system and, by continuity, as $\varepsilon \to 0$ we can obtain a maximum delay margin estimate arbitrary close to the estimate obtained with $\phi_{\bar{\tau}}$.

\subsection{Obtaining a rational $\tilde T$}
\label{subsec:rational_T_tilde}

Selecting $\bar \tau$ to be the supremum for which \eqref{eq:new_interpolation} holds gives rise to a corresponding Pick matrix \eqref{eq:pick_matrix} that is singular and a unique solution $\tilde T$ which is a Blaschke product \cite[pp.~5-9]{garnett2007bounded}. For such a solution, $\| \tilde T \|_{\Hinf} =1$ and thus it does not satisfy \eqref{eq:new_interpolation_a}. Therefore, the corresponding $T$ obtained via \eqref{eq:Ttilde2T} may not have delay margin $\bar\tau$. However, for any $\bar \tau$ smaller than the supremum the Pick matrix is positive definite and the analytic interpolation problem \eqref{eq:new_interpolation} has infinitely many rational solutions \cite{byrnes2001ageneralized}, \cite{fanizza2007passivity}. One of these is the so-called maximum entropy solution \cite{georgio2003kullback}, \cite{blomqvist2005optimization}, which can be obtained by the Nevanlinna-Pick algorithm \cite{Walsh1956book}.

\subsection{Model reduction of the weight for rational control implementation} \label{sec:quasi_convex}

Given a function $\tilde T$ that solves \eqref{eq:new_interpolation}, from \eqref{eq:def_T} and \eqref{eq:Ttilde2T} the stabilizing controller is given by
\begin{equation}
\label{eq:K}
K=P^{-1}(-\tilde{T} +(1-T_0)W_{\bar{\tau}})^{-1}(\tilde{T} +T_0W_{\bar{\tau}}).
\end{equation}
However, $\outerfuncwzero$,  defined by  \eqref{eq:outerrepr}, will typically not be rational. Therefore, from \eqref{eq:K} we see that, even if $\tilde T$ and $T_0$ are rational, $K$ is generally not. To overcome this, we propose to do rational approximation of the function $\outerfuncwzero$ by using quasi-convex optimization as in \cite{karlsson2008stability-preserving}. We can still guarantee that the 
\rev{rational approximation can be used to compute a lower bound on the delay margin} by enforcing that the magnitude of the obtained rational approximation is an overestimate of $|\outerfuncwzero(i\omega)|$, i.e., of $\phi_{\bar \tau}$.
This can be formulated as
the problem to find a triplet $(\epsilon, a,b)\in \mR_+\times\mS_{n_a}\times\mS_{n_b}$ which minimizes $\epsilon$ subject to the constraint
\begin{equation}\label{eq:ba_bound}
\phi_{\bar \tau}(\omega)\le \left|\frac{b(i\omega)}{a(i\omega)}\right| \le (1+\epsilon)\phi_{\bar \tau}(\omega), \quad \forall \,\omega \in \mR ,
\end{equation}
where $\mS_n$ is the set of stable%
\footnote{A polynomial is stable if it is nonzero in $\mC_+$.}
polynomials of degree bounded by $n$. Note that this may be written as
\begin{align}\label{eq:quasiconvex}
\min_{\substack{\epsilon \in \mR_+ \\ A \in \mT_{n_a} \\ B \in\mT_{n_b}}} \quad & \epsilon\\
\mbox{ subject to}\quad &  \phi_{\bar \tau}(\omega)^2A(i \omega)\le B(i\omega) \le (1+\epsilon)^2\phi_{\bar \tau}(\omega)^2A(i \omega)\nonumber\\ 
&\mbox{for } \omega\in \mR\nonumber,
\end{align}
where $\mT_n$ is the set of nonnegative trigonometric polynomials of degree bounded by $n$. Note that the nonnegatvity of $B$ and $A$ can be enforced by using an LMI formulation \cite[Thm.~2.5]{dumitrescu2007positive}. After solving \eqref{eq:quasiconvex}, one can recover the rational approximation of $\outerfuncwzero$ as $\outerfuncwzerotilde:=b/a$ where $b$ and $a$ are the stable spectral factors of $B$ and $A$, respectively. The optimization problem \eqref{eq:quasiconvex} is quasi-convex since the feasibility problem is convex for any fixed $\epsilon$, and it can be thus be solved by, e.g., bisection over $\epsilon$.

By using the rational weight $\outerfuncwzerotilde$ instead of $\outerfuncwzero$ \rev{when solving problem \eqref{eq:new_interpolation}}, the resulting control system obtained in \eqref{eq:K} will be of bounded degree and thus possible to implement using standard methods. \rev{Note that if $\bar \tau$ is such that \eqref{eq:new_interpolation} has a solution, there is an $\epsilon>0$ so that \eqref{eq:new_interpolation} also has a solution for any outer function $\tilde{W}_{\bar\tau}=b/a$ satisfying \eqref{eq:ba_bound}. Since any given error bound $\epsilon>0$ can be achieved by selecting $n_a$ and $n_b$ large enough, this approach can always be used to achieve a rational controller with delay margin arbitrary close to the supremum $\bar \tau$ for which \eqref{eq:new_interpolation} has a solution.}  

\rev{
\begin{remark}
The idea of overestimating the weight $\phi_{\bar\tau}$ in \eqref{eq:w_tau} was used in \cite[Sec.~3]{wang1994representation}, \cite[Sec.~II.C]{qi2017fundamental}, and closed-form rational approximations have been derived. However, these expressions are only valid for $T_0(s) \equiv 0$ in \eqref{eq:weight_with_T0} and for $T_0(s) \not\equiv 0$ we therefore need the method described above.
\end{remark}
}

\subsection{Considerations for strictly proper $T$}
\label{sec:strict_proper_T}

By Lemma~\ref{lem:T_0_infty}, we must have that $\real(T_0(\infty)) < 1/2$. However, this condition is not enough to ensure that it is possible to obtain a $T$ that corresponds to a strictly proper closed loop system, i.e., a $T$ such that $T(\infty) = 0$. To ensure the latter we need a $T_0$ such that $|T_0(\infty)| \leq |1/2 - \real(T_0(\infty))|$, or equivalently such that $\imag(T_0(\infty))^2 + \real(T_0(\infty)) \leq 1/4$. To ensure that indeed $T(\infty) = 0$, we need $T_0(\infty) = 0$ and to modify the weight $\phi_{\bar{\tau}}$ so that it enforces a ``roll-off'' at an appropriate speed, i.e., modify $\phi_{\bar{\tau}}$ so that for sufficiently large frequencies it goes to infinity with the appropriate speed. Finally, note that this type of ``roll-off'' can also be achieved by selecting an improper approximation in Section~\ref{sec:quasi_convex}, in which case the upper bound in the approximation needs to be removed for large enough frequencies.

\section{Iterative improvement of $T_0$}
\label{sec:homotopy}

In Section~\ref{sec:improved_method} we introduced a nominal complementary sensitivity function $T_0$ as the center of a frequency dependent disc containing $T(i\omega)$ for each $\omega$. However, how to select a $T_0 \not \equiv 0$ in order to achieve a better lower bound on the maximum delay margin is a nontrivial question. In this section we describe one iterative heuristic for selecting $T_0$.

To this end, we note that if we can solve the interpolation problem \eqref{eq:suff_cond}, then for any $\alpha \in [0, 1]$ we can also solve the interpolation problem
\begin{align*}
\inf_{T \in \Hinf} & \quad \|T(i \omega) \phi_{\bar{\tau}}(\omega)  \|_{\Linf} \\
\text{subject to}  & \quad T(p_j) = \alpha,\quad j = 1,\ldots, n , \\
                   & \quad T(z_j)= 0,\quad j = 1,\ldots, m.
\end{align*}
In fact, if $T$ solves \eqref{eq:suff_cond}, then $\alpha T$ solves the above problem, and $\|\alpha T(i \omega) \phi_{\bar{\tau}}(\omega)  \|_{\Linf} \leq \|T(i \omega) \phi_{\bar{\tau}}(\omega)  \|_{\Linf} < 1$. Moreover, for $\alpha = 0$, $T \equiv 0$ is a solution. This means that $\alpha$ can be interpreted as a homotopy parameter. Using this insight, we propose the heuristic method in Algorithm~\ref{alg:T0_iterative} for selecting $T_0$ and for computing a lower bound on the maximum delay margin. Note that step 5 involves a rational approximation, which, however, is still guaranteed to fulfill the interpolation conditions in step 4. This approximation method, which is quasi-convex, was developed for the discrete time setting in \cite{karlsson2008stability-preserving}, and for completeness it is described below in subsection~\ref{sec:rational_approx_interpolant} in our present continuous-time setting. The approximation step is done in order to keep the degree of the solution constant, independently of the number of steps $N$, in the homotopy method.

\renewcommand{\algorithmicrequire}{\textbf{Input:}}
\renewcommand{\algorithmicensure}{\textbf{Output:}}

\begin{algorithm}[tb]
\caption{Iterative heuristic method for selecting $T_0$ and computing a lower bound on the maximum delay margin.}
\label{alg:T0_iterative}
\begin{algorithmic}[1]
\REQUIRE Unstable poles $p_j$,  $j = 1, \ldots, n$, and nonminimum phase zeros $z_j$,  $j = 1, \ldots, m$, of the plant $P$. Limits $n_b$ and $n_a$ on the degree of the numerator and denominator, respectively, of the rational approximation of the weight $\tilde \outerfuncwzero$. Limit $n_{T} \geq n+m$ on the degree of $T_0$. Initial guess $\hat{\tau}_+^{(0)}$. Number of homotopy steps $N$. Initial $T^{(0)}$.
\STATE $\tau_- = 0.$%, $\quad \tau_{\rm mid} = (\tau_+ + \tau_-)/2$
\FOR{$k = 1, \ldots, N$}
\STATE $\alpha^{(k)} = k/N$, \quad $T_0^{(k)} = T^{(k - 1)}$, \quad $\tau_+^{(k)} = \hat{\tau}_+^{(k-1)}$
\STATE Use modified version of Algorithm~\ref{alg:first_method} from Section~\ref{sec:improved_method}, initialized with the upper bound $\tau_+^{(k)}$ and in each step using a rational approximation $\outerfuncwzerotilde$ of the weight function $\outerfuncwzero$ computed as outline in Section~\ref{sec:quasi_convex}, to solve
\vspace{-6pt}
\begin{align*}
\max_{\substack{\tau \, \in \, \mR_+ \\ T \, \in \, \Hinf}} & \quad \tau \\
\text{s.t.} & \quad \| (T - T_0^{(k)}) \phi_{\tau} \|_{\rev{\Linf}} < 1 \\
& \quad T(p_j) = \alpha^{(k)}, \text{ for } j = 1, \ldots, n, \\
& \quad T(z_j) = 0, \text{ for } j = 1, \ldots, m.
\end{align*}
Let $\hat{\tau}_+^{(k)}$, $T^{(k - 1/2)}$ be the maximizing arguments.
\STATE $T^{(k)} =$ rational approximation, with degree bounded by $n_{T}$, of the interpolant $T^{(k - 1/2)}$, computed as described in Section~\ref{sec:rational_approx_interpolant}.
\ENDFOR
\STATE $\bar{\tau} = \hat{\tau}_+^{(N)}$, \quad $T = T^{(N - 1/2)}$, \quad $T_0 = T_0^{(N)}$.
\ENSURE $\bar{\tau}$, \; $T$, \; $T_0$.
\end{algorithmic}
\end{algorithm}

\begin{remark}
In the examples to be presented in Section~\ref{sec:example}, Algorithm~\ref{alg:T0_iterative} sometimes encounters numerical problems. A remedy to this seems to be to change step 5 in Algorithm~\ref{alg:T0_iterative} to
\begin{align*}
\tilde{T}^{(k)} & = \text{rational approximation of the interpolant } T^{(k - 1/2)}, \\
& \phantom{=} \text{ computed as described in Section~\ref{sec:rational_approx_interpolant}}, \\
T^{(k)} & = \gamma \tilde{T}^{(k)},
\end{align*}
for some $\gamma \in (0, 1]$. This seems to improve the conditioning when computing the weight approximation $\outerfuncwzerotilde$ in later iterations.
\end{remark}

\subsection{Rational approximation of interpolant}
\label{sec:rational_approx_interpolant}

Without the degree reduction in step 5 of Algorithm 2 the degree of $T$ would increase in each iteration. To keep the degree  bounded we use a continuous-time version of the procedure introduced in \cite{karlsson2008stability-preserving}, which allows for preserving both a set of interpolation conditions as well as the stability of the complementary sensitivity function. The procedure is based on the optimization problem
\begin{subequations}\label{eq:opt2}
\begin{align}
\min_{T\in \rev{\Hinf}(\mC_+)} &\|\sigma T\|_{\HtwoC}\label{eq:opt2a}\\
\mbox{ subject to } &T(s_k)=w_{k}, \mbox{ for } k=0,\ldots, n,\label{eq:opt2b}
\end{align}
\end{subequations}
where $\sigma$ is an outer weight. A function $T\in \HinfC$ is a minimizer of \eqref{eq:opt2} if it satisfies \eqref{eq:opt2b} and $\sigma T \in \script{K}$, where 
\begin{align}
\script{K}:=\left\{\sum_{k=0}^n \alpha_k \frac{1}{s+\bar s_k} \; \Big| \; \alpha_k\in \mC,\, \mbox{ for }k=0,\ldots, n\right\}\label{eq:coinvariant}
\end{align}
is the co-invariant subspace. This follows by representing \eqref{eq:opt2b} as an integral constraint using Cauchy integral formula, see, e.g., \cite[Thm.~10.15]{rudin1987real}, and then applying the projection theorem \cite[p.~65, Thm.~2]{luenberger1969optimization}. For the discrete time counterpart, i.e., the corresponding problem on the unit disc $\mD$, see \cite[Sec.~III]{karlsson2008stability-preserving}.

Next, we would like to use the weight $\sigma$ as a tuning parameter in order to achieve solutions $T$ of reduced order. To this end, let $\sigma$ be an outer function and let $T$ be the corresponding solution to \eqref{eq:opt2}. Then by selecting an outer weight $\hat \sigma$ in a suitable class and close to $\sigma$, we can guarantee that the optimal solution $\hat T$ of \eqref{eq:opt2} with weight $\hat \sigma$ is close to $T$ and is of low degree.

In particular, given $T$ satisfying \eqref{eq:opt2b} let $\sigma$ be an outer function such that $\sigma T\in \script{K}$. Such $\sigma$ is always possible to find  if the inner part of $T$ is a Blaschke product of degree bounded by $n$ (see \cite[Thm.~6]{karlsson2008stability-preserving}). Then $\sigma$  can be selected as any function $a/T$ where $a\in \script{K}$ has zeros in all zeros of $T$ in $\mC_+$ (including multiplicity). In order to find a low degree approximation $\hat T$ of $T$ we proceed by finding a rational approximating $\hat \sigma$ of $\sigma$. This can be done via the quasi-convex optimization problem as in Section~\ref{sec:quasi_convex}. Finally we solve \eqref{eq:opt2} using the weight $\hat \sigma$ in order to get the approximation $\hat T$. Note that if the weight is of the form $\hat \sigma=\hat \sigma_1\hat \sigma_2$ where $\hat \sigma_1\in \script{K}$, then the degree of $\hat T$ is bounded by $n+\deg(\hat \sigma_2)$. The following theorem ensures that this is a good approximation if the weight $\hat \sigma$ is sufficiently close to $\sigma$.

\begin{thm}\label{thm:approximation}
Let $\sigma\in \HinfC$ be an outer function such that $|\sigma(i\omega)|$ is bounded away from zero on $\mR$, and let $T$ be the corresponding solution to \eqref{eq:opt2}. Moreover, let $(\sigma_\ell)_{\ell=1}^\infty$ be a sequence of outer functions and let $(T_\ell)_{\ell=1}^\infty$ be the corresponding optimal solutions to \eqref{eq:opt2}. Then $T_\ell\to T$ in $\HinfC$ as $\ell\to \infty$ if
\begin{align}
&|\sigma_\ell(i\omega)|\to |\sigma(i\omega)|\mbox{ in } {\Linf(\mR)} \mbox{ as } \ell\to \infty\label{eq:sigma1}
\end{align}
and there exists a constant $M$ such that
\begin{subequations}\label{eq:derbound}
\begin{align}\left\|\frac{\partial |\sigma(i\omega)|^2}{\partial \omega}\right\|_{\Ltwo(\mR,(1+\omega^2)d\omega)}&\le M\label{eq:derbound1}\\
\left\|\frac{\partial |\sigma_\ell(i\omega)|^2}{\partial \omega}\right\|_{\Ltwo(\mR,(1+\omega^2)d\omega)}&\le M, \mbox{ for } \ell=1,2,\ldots .\label{eq:derbound2}
\end{align}
\end{subequations}
\end{thm}

\begin{proof}
See Appendix~\ref{app:approximation}.
\end{proof}

\section{Optimization for multiple stability margins}
\label{sec:sim_margins} 

When optimizing one stability margin, it is common that other stability margins deteriorate. Therefore, in control design it is natural to impose criteria on several stability margins. A first approach to this is to require that each of the margins satisfy given lower bounds. This amounts to finding a complementary sensitivity function that avoids the union of the forbidden areas for the stability margins.

We can extend our method to compute a lower bound for the maximum delay margin, given fixed lower bounds $k$ and $\varphi$ on the gain and phase margin, respectively. This corresponds to disturbances $\Delta$ in the set
\[
\{ \kappa \mid \kappa \in [1, k] \} \cup \{ e^{-i\theta} \mid \theta \in [-\varphi, \varphi] \} \cup \{ e^{-s\tau} \mid \tau \in [0, \bar{\tau}] \},
\]
with corresponding forbidden areas depicted in Figure~\ref{subfig:nyquist_gain_phase_delay_indep}. To this end, define the corresponding weight functions, representing the inverse of the distance from $T_0(i\omega)$ to the respective forbidden area, as
\begin{subequations}
\begin{align}
\phi^{\rm gain}_k(\omega)^{-1} & = \dist([-\infty,-1/(k-1)], T_0(i\omega)) \label{eq:phi_gain} \\
\phi^{\rm phase}_{\varphi}(\omega)^{-1} & = \dist(\cut{\varphi} \cup \cut{-\varphi}, T_0(i\omega)) \label{eq:phi_phase} \\
\phi^{\rm delay}_{\bar{\tau}}(\omega)^{-1} & = \dist(\cut{\omega\bar{\tau}}, T_0(i\omega)) \label{eq:phi_delay}
\end{align}
\end{subequations}
where we for clarity use $\phi^{\rm delay}_{\bar{\tau}}(\omega)$ to denote the weight \eqref{eq:phi_T_zero} corresponding to the delay margin. The smallest distance to the union of the forbidden areas is represented by the maximum of the three functions, i.e.,
\begin{equation}\label{eq:indep_margin_weight}
\phi_{\bar{\tau}}(\omega) = \max(\phi^{\rm gain}_{\kappa}(\omega), \phi^{\rm phase}_{\varphi}(\omega), \phi^{\rm delay}_{\bar{\tau}}(\omega)),
\end{equation}
and as before, $T_0(i\omega)$ must not intersect any of these areas. The procedure is now the same as in Section~\ref{sec:improved_method}, i.e., to iteratively check solvability of \eqref{eq:new_interpolation} where now $W_{\bar{\tau}}$ is the outer extension of $\phi_{\bar{\tau}}(\omega)$, as defined in \eqref{eq:indep_margin_weight}, via \eqref{eq:outerrepr}.

This procedure is straightforward to implement. However, as noted in Section~\ref{subsec:gain_and_phase_margin} this may lead to feedback designs that are not robust to simultaneous disturbances. The latter corresponds to disturbances $\Delta$ in the set
\[
\left\{ \kappa e^{-i\theta} e^{-s\tau} \mid \kappa \in [1, k], \; \theta \in [-\varphi, \varphi], \; \tau \in [0, \bar{\tau}] \right\}.
\]
For the complementary sensitivity function this results in the forbidden region 
$\mathfrak{D}(\omega)$ given by
\[
\left\{\frac{\kappa e^{-i\theta} e^{-i\omega\tau}}{1 + \kappa e^{-i\theta} e^{-i\omega\tau}}    \mid    \kappa \in [1, k], \ \theta \in [-\varphi, \varphi],  \tau \in [0, \bar{\tau}] \right\}.
\]
The forbidden region $\mathfrak{D}(\omega)$ is considerably larger than the one obtained with independent margins, as can be seen by comparing Figures~\ref{subfig:interpolation_gain_phase_delay_indep} and \ref{subfig:interpolation_gain_phase_delay}.
The corresponding weight function is then 
\[
\phi_{\bar{\tau}}(\omega)^{-1} = \dist\left(\mathfrak{D}(\omega), T_0(i\omega)
\right),
\]
and we can apply the proposed method also in this setting.
To compute $\dist\left(\mathfrak{D}(\omega), T_0(i\omega)\right)$, first note that it can be  easily checked if $T_0(i\omega) \in \mathfrak{D}(\omega)$. This is done in the Nyquist domain by checking if $\Gamma \circ T_0(i\omega) \in \Gamma \circ \mathfrak{D}(\omega)$, where $\Gamma(s) = s/(1-s)$ is the inverse of \eqref{eq:mobius} \rev{and $\circ$ denotes function composition}. Indeed, $\Gamma \circ \mathfrak{D}(\omega)$ has a simple representation in polar coordinates, cf. Figure~\ref{subfig:nyquist_gain_phase_delay}. If $T_0(i\omega) \not\in \mathfrak{D}(\omega)$, then the distance can be computed as the minimum of distances to each of the five arcs and lines that define the boundary of $\mathfrak{D}(\omega)$. Expressions for these can be derived from the equations in Appendix~\ref{app:sim_gain_phase}, and are left out for the sake of brevity.

\begin{remark}
A notable difference between the two types of aggregate margins is that for the independent margins the cut corresponding to the delay margin is contained in the cuts corresponding to the phase margin for sufficiently small frequencies, i.e., $\cut{\omega \bar{\tau}} \subset \cut{\varphi} \cup \cut{-\varphi}$   for $|\omega | \leq \varphi/\bar{\tau}$.
By contrast, for the simultaneous margins, the uncertainties are multiplied and thus the forbidden area is the  point-wise product of the three independent forbidden areas. 
This is illustrated in Figures~\ref{subfig:nyquist_gain_phase_delay} and \ref{subfig:interpolation_gain_phase_delay}.
\end{remark}

\begin{figure}[tb]
\begin{center}
\begin{subfigure}{.49\columnwidth}
 \centering
 \includegraphics[width=\textwidth]{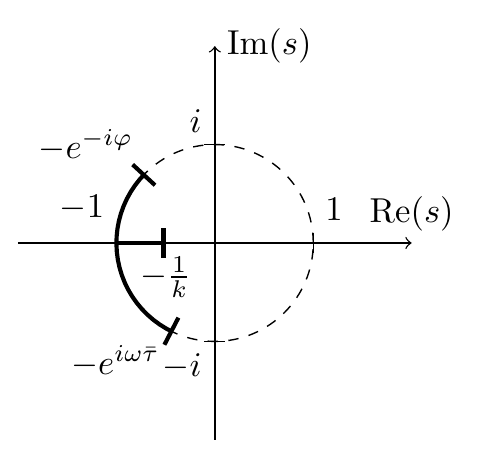}
 \subcaption{Forbidden areas in the Nyquist plot for independent margins.}
 \label{subfig:nyquist_gain_phase_delay_indep}
\end{subfigure}
\hfill
\begin{subfigure}{.49\columnwidth}
 \centering
 \includegraphics[width=\textwidth]{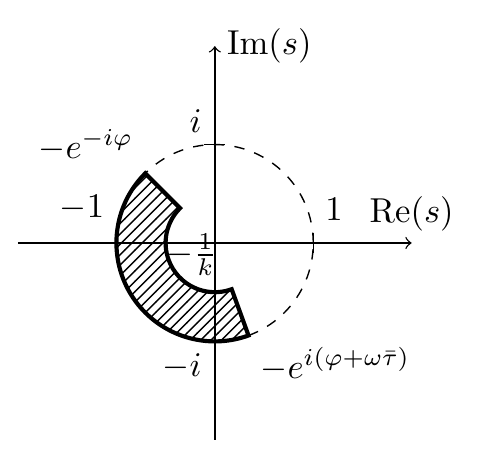}
 \subcaption{Forbidden areas in the Nyquist plot for simultaneous margins.}
 \label{subfig:nyquist_gain_phase_delay}
\end{subfigure}
\\
\begin{subfigure}[t]{.49\columnwidth}
 \centering
 \includegraphics[width=\textwidth]{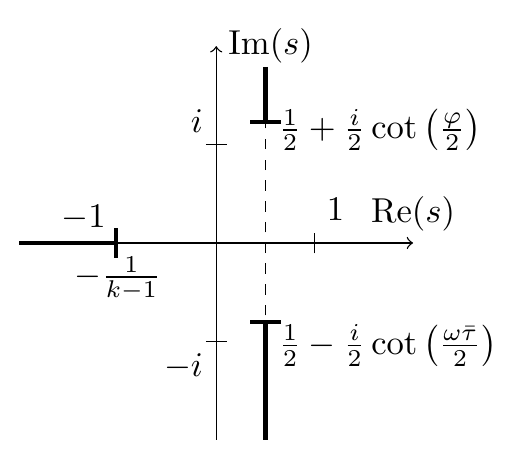}
 \subcaption{Forbidden areas in the range of the interpolant for independent margins.}
 \label{subfig:interpolation_gain_phase_delay_indep}
\end{subfigure}
\hfill
\begin{subfigure}[t]{.49\columnwidth}
\centering
  \includegraphics[width=\textwidth]{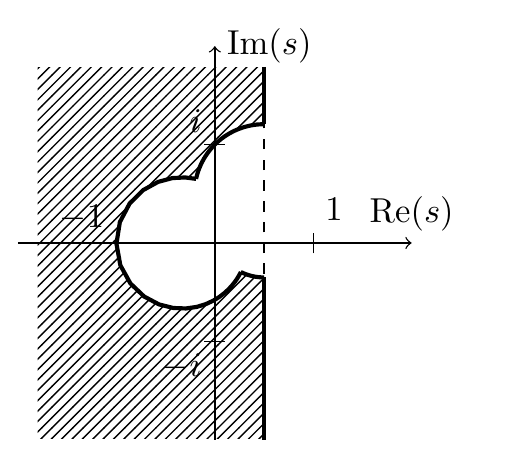}
  \subcaption{Forbidden areas in the range of the interpolant for simultaneous margins.}
 \label{subfig:interpolation_gain_phase_delay}
\end{subfigure}
\caption{Illustration of gain, phase, and delay margins. Left column shows forbidden areas for independent margins, in the Nyquist plot and for the range of the interpolant, and right column shows forbidden areas for simultaneous margins. The margins are illustrated for one frequency $\omega > 0$ such that $\varphi < \omega \bar{\tau} < 2\pi - \varphi$.}\label{fig:nyquist_interpolation_gain_phase_delay_margin}
\end{center}
\end{figure}

\rev{
\begin{remark}
Note that the approximation method for $\phi_{\bar \tau}(\omega)$ descried in Section~\ref{sec:quasi_convex} can be used in the same way on the problem with multiple stability margins to obtain a rational approximation of the weight function and hence a finite degree controller. Therefore, the homotopy method in Section~\ref{sec:homotopy} can also be extended to this setting.
\end{remark}
}%

\section{Numerical example}
\label{sec:example}

In this section we present some examples to investigate the performance of the methods proposed in Sections~\ref{sec:improved_method} and \ref{sec:homotopy}. To facilitate comparison with previous results we consider the various SISO-systems given in \cite[Ex.~1]{qi2017fundamental}. Moreover, in the last example of the section we also illustrate how to use the results from Section~\ref{sec:sim_margins} in order to design a controller with given constraints on minimum simultaneous gain and phase margins while optimizing the delay margin.

In all examples below, we use the techniques described in Sections~\ref{subsec:pole_in_zero} and \ref{subsec:rational_T_tilde}. In particular, we take $\varepsilon = 10^{-4}$ as in \eqref{eq:phi_lowerbound}, and we use a degree ten rational function to approximate the weight $\phi_{\bar{\tau}}(\omega)$. Moreover, when using the homotopy method in Algorithm~\ref{alg:T0_iterative}, we make three steps in the method (i.e., we set $N = 3$) and make a degree ten rational approximation of the obtained interpolant in each step, as described in Section~\ref{sec:rational_approx_interpolant}. Finally, we use the code associated with \cite{blomqvist2005optimization} to compute the maximum entropy solution to the interpolation problems.

\subsection{A first example}\label{sec:first_example}
We consider the system
\begin{equation}\label{eq:chen_ex_1_3}
P(s)=0.1\frac{(0.1s-1)(s+0.1659)}{(s-0.1081)(s^2+0.2981s+0.06281)},
\end{equation}
i.e., \cite[eq.~(42)]{qi2017fundamental}, which has one real unstable pole ($p=0.1081$) and one real nonminimum-phase zero ($z=10$). Similar systems have also been considered before in the literature on delay systems \cite[eq.~(22)]{middleton2007achievable}, \cite[Sec.~5]{michiels2002continuous}. In this example, since the unstable pole is closer to origin than the nonminimum phase zero, \cite[Rem.~17]{middleton2007achievable} implies that $2/p - 2/z = 2/0.1081 - 2/10 \approx 18.3$ is a tight bound on the maximum delay margin.

We use the method in Section~\ref{sec:improved_method} to compute a lower bound on the maximum delay margin and take $T_0(s) \equiv T_0$ real and constant. We vary $T_0$ in the interval $[-50, 1/2)$ and also compare with the method in \cite{qi2017fundamental}. The results are presented in Figure~\ref{fig:ex1}. As can be seen from the figure, for appropriate choices of $T_0$ our method outperforms that of \cite{qi2017fundamental}.

\begin{remark}
As $T_0$ tends to $-\infty$, the result obtained with our method  converges to the true maximum delay margin. This can be understood by analyzing the proof in \cite{middleton2007achievable} which shows that the bound is tight. In particular, it follows by using \eqref{eq:mobius} to transform the argument in \cite[Rem.~17]{middleton2007achievable} from the Nyquist domain to the interpolation domain and by considering the feasible regions for the corresponding interpolants.
\end{remark}

\begin{figure}[tb]
  \centering
  \includegraphics[width=0.475\textwidth]{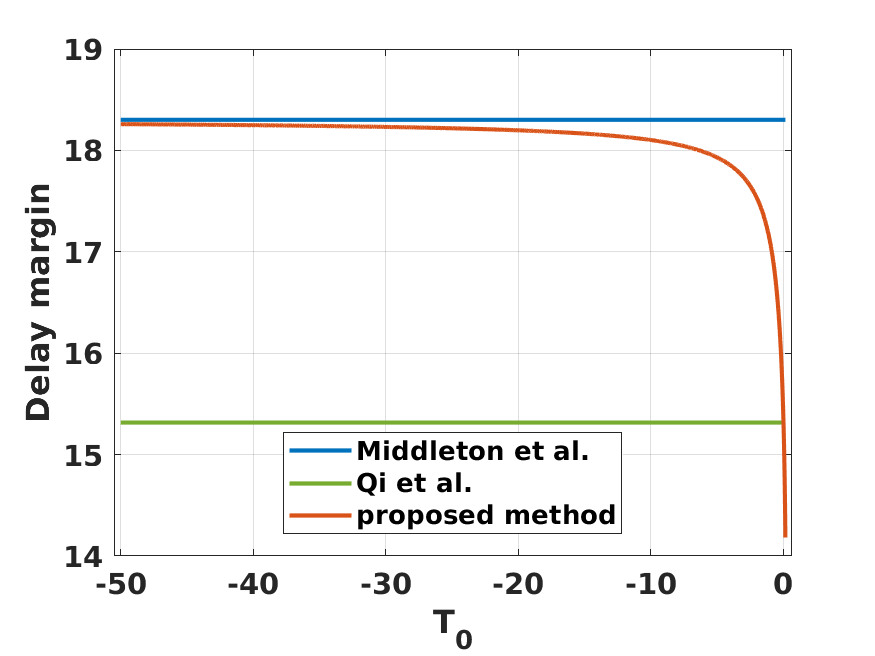}
  \caption{Results for the example in \eqref{eq:chen_ex_1_3}, with $T_0(s) \equiv T_0$ real and constant. When $T_0$ goes to $-\infty$ we get arbitrarily close to the tight bound from \cite{middleton2007achievable}, while for $T_0 > 0$ the bound deteriorate quickly.}
  \label{fig:ex1}
\end{figure}

\subsection{Systems with one unstable pole and one nonminimum phase zero}

The system \cite[eq.~(41)]{qi2017fundamental}, i.e.,
\begin{equation}\label{eq:chen_ex_1_2}
P(s)=\frac{s-z}{s-p},
\end{equation}
where $z,p > 0$, has a similar characteristic as the previous example since it has exactly one real unstable pole and one real nonminimum phase zero. Following \cite{qi2017fundamental}, we fix $z=2$ and compute an estimate for the delay margin for different values of $p$ in the interval $[0.3, 4]$. We use the method in \cite{qi2017fundamental}, the method in Section~\ref{sec:improved_method} with $T_0(s) \equiv T_0$ real and constant, and the homotopy method in Section~\ref{sec:homotopy}. The results are shown in Figure~\ref{fig:ex2}.

As can be seen in Figure~\ref{fig:ex2}, in the region $p < z = 2$, where the theoretical upper bound from \cite{middleton2007achievable} is tight, we get a considerable improvement over the bound in \cite{qi2017fundamental} by taking $T_0 = -10$, cf. Section~\ref{sec:first_example}. In this region, the homotopy method also perform better than \cite{qi2017fundamental}, however it does not achieve the upper bound. In the region $p > z = 2$ where, to the best of our knowledge, the true stability margin is still unknown, the choice of constant $T_0 = -10$ perform worse than \cite{qi2017fundamental}. On the other hand, in this region the value $T_0=0.35$ achieves some improvement. However, the best lower bound in this region is obtained by the homotopy method.

\begin{figure}[tb]
  \centering
  \includegraphics[width=0.475\textwidth]{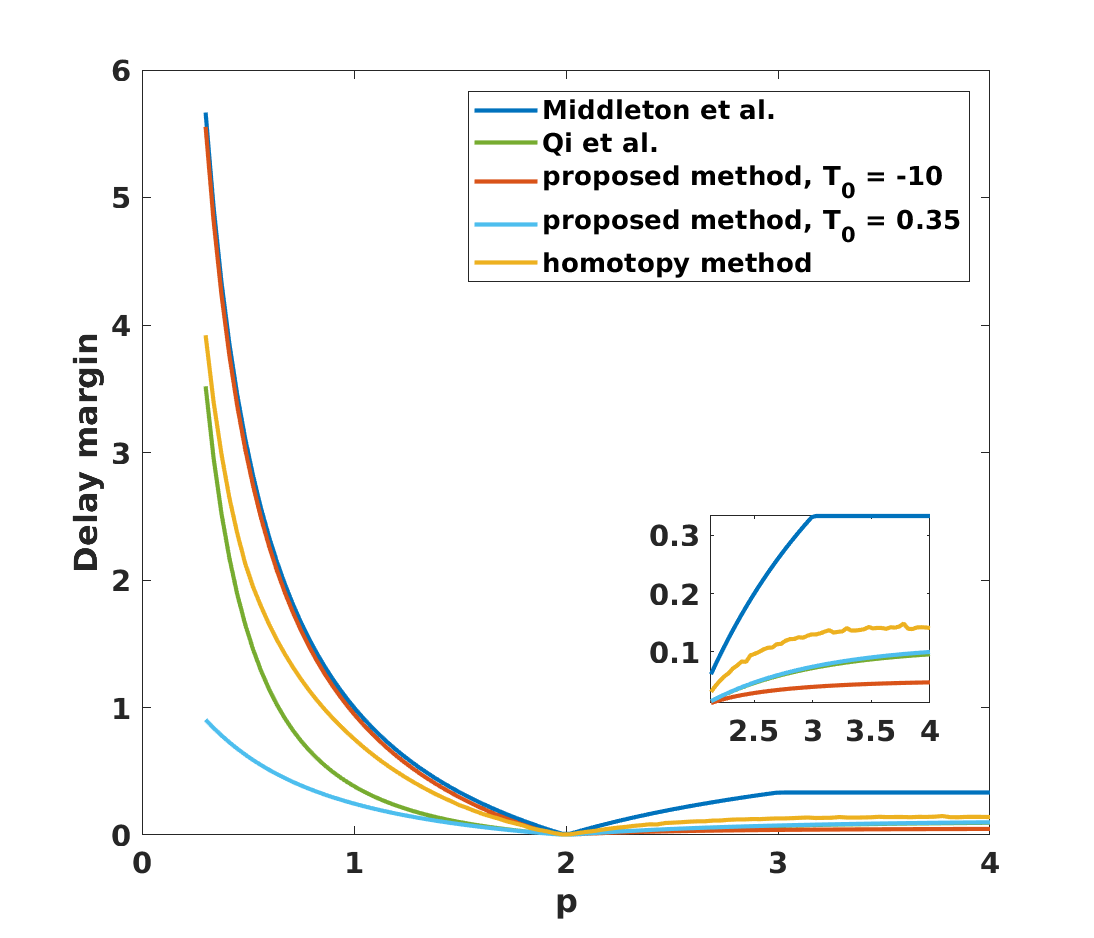}
  \caption{Results for the example in \eqref{eq:chen_ex_1_2}.}
  \label{fig:ex2}
\end{figure}

\subsection{System with two unstable real poles}

Next we consider the system \cite[eq.~(40)]{qi2017fundamental}, which has two distinct real unstable poles. The system is given by
\begin{equation}\label{eq:chen_ex_1_1}
P(s)=\frac{1}{(s-p_1)(s-p_2)},
\end{equation}
where $p_1, p_2 > 0$. We fix $p_1$ to $0.2$, and the delay margin is computed for different values of $p_2 \in [0.1, 3]$ using the method in Section~\ref{sec:improved_method} with $T_0(s) \equiv T_0$, real and constant, the method in  Section~\ref{sec:homotopy}, and the method in \cite{qi2017fundamental}. In the first method, for each value of $p_2$ we let the constant $T_0$ vary in the interval $[-10, 0.5)$ (with steps of size $0.1$) and take the best of these values as the bound.

Compared to the (not necessarily tight) upper bound from \cite{middleton2007achievable}, all three methods perform similarly and the gap is quite large. For a fixed $T_0$, the improvements are between $0.19\%$ and $2.8\%$ compared to the method in \cite{qi2017fundamental}, depending on $p_2$. Similarly, the homotopy method perform between $7.3\%$ better and $16\%$ worse than the method in \cite{qi2017fundamental}. The complete results are left out due to space limitations.

\begin{figure}[tb]
 \centering
 \includegraphics[width=0.475\textwidth]{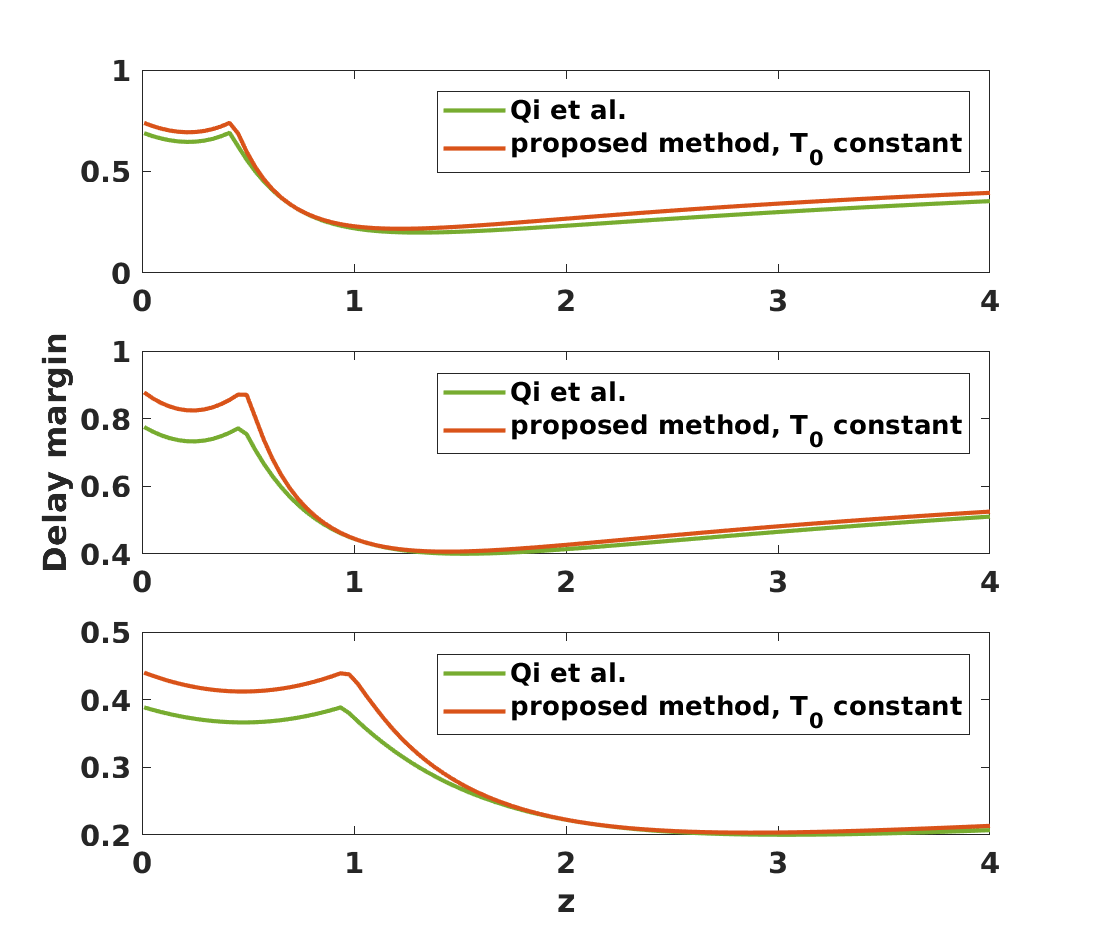}
 \caption{Reults for the example in \eqref{eq:chen_ex_1_4}. Estimates of the delay margin for,
 % the cases,
 from top to bottom,  $(r, \theta) = (1, \pi/4)$, $(r, \theta) = (1, \pi/3)$, and  $(r, \theta) = (2, \pi/3)$.}
 \label{fig:ex4}
\end{figure}

\subsection{System with conjugate pair of complex poles}

We now consider the system \cite[eq.~(45)]{qi2017fundamental}, given by 
\begin{equation}\label{eq:chen_ex_1_4}
P(s)=\frac{s-z}{(s-re^{i\theta})(s-re^{-i\theta})},
\end{equation}
where $z, r > 0$ and $\theta \in [0, \pi/2]$. The system thus has a pair of unstable complex poles in $re^{\pm i \theta}$ and a nonminimum phase zero in $z$. We fix three different sets of complex poles, namely $(r, \theta) = (1, \pi/4)$, $(r, \theta) = (1, \pi/3)$, and $(r, \theta) = (2, \pi/3)$, and vary $z$ in $[0.01, 4]$. For each position of the zero $z$ we compute estimates of the delay margin. Similar to the previous example, for the method described in Section~\ref{sec:improved_method} we let $T_0(s) \equiv T_0$ be real and constant, but for each value of $z$ we vary $T_0$ in the interval $[-1.5, 0.5)$ (with steps of size $0.02$) and take the best of these values as the bound. The results are shown in Figure~\ref{fig:ex4}
As can be seen in the figure, the proposed method gives significantly improved bounds in some regions, for example when $\theta=\pi/3$ and $z$ is small compared to $r$.

\subsection{Control design for increased delay margin with simultaneous gain and phase margins}

We conclude this section with an example where we design a controller by maximizing the lower bound of the delay margin while at the same time ensuring simultaneous gain and phase margins. To this end, we revisit the plant \eqref{eq:chen_ex_1_3} in Section~\ref{sec:first_example}, with poles in $\{0.1081, -0.1490 \pm i 0.2015 \}$ and zeros in $\{10, -0.1659\}$. We specify a simultaneous gain and phase margin of $1.5$ and $\pi/12$, respectively, and use the method described in Section~\ref{sec:sim_margins} to maximize the simultaneous delay margin.

Since the plant is strictly proper, in order to obtain a controller that is proper we need $T$ to be strictly proper. We therefore take $T_0 \equiv 0$ and let the approximation $\outerfuncwzerotilde$ be improper with relative degree one, cf. Section~\ref{sec:strict_proper_T}. Using this setup, the algorithm returns a closed-loop system that is guaranteed to achieve a simultaneous gain, phase and delay margin of $1.5$, $\pi/12$, and $1.870$, respectively. However, it can be verified that the simultaneous margins of the returned system are in fact $1.5$, $\pi/12$, and $2.254$, respectively,%
\footnote{The independent margins achieved are an independent gain margin of $4.629$ at infinite frequency, an independent phase margin of $0.5248 \pi$ at phase crossover frequencies $\pm \; 0.4107$, and an independent delay margin of $4.015$.}
i.e., the algorithm returns a conservative estimate. This conservativeness is most likely due to the approximation errors in the approximation of $\outerfuncwzero$. The obtained complementary sensitivity function $T$ is given in \eqref{eq:T_controller_design_ex}, and a plot of $T(i\mR)$ is shown in Figure~\ref{subfig:ex_controller_design_T}. The critical frequencies $\omega_c$, where $T(i\omega_c)$ touches the forbidden region, are $\omega_c = \pm \; 0.7188$. The corresponding open-loop system, given by $P(s)K(s) = T(s)/(1 - T(s))$, is given in \eqref{eq:PK_controller_design_ex}, and the Nyquist plot is shown in Figure~\ref{subfig:_controller_design_Nyquist}. Finally, the obtained controller $K(s)$ is given in \eqref{eq:K_controller_design_ex}.

\begin{figure*}[tb]
\begin{center}
\begin{subfigure}[t]{.49\textwidth}
 \centering
 \includegraphics[width=\textwidth]{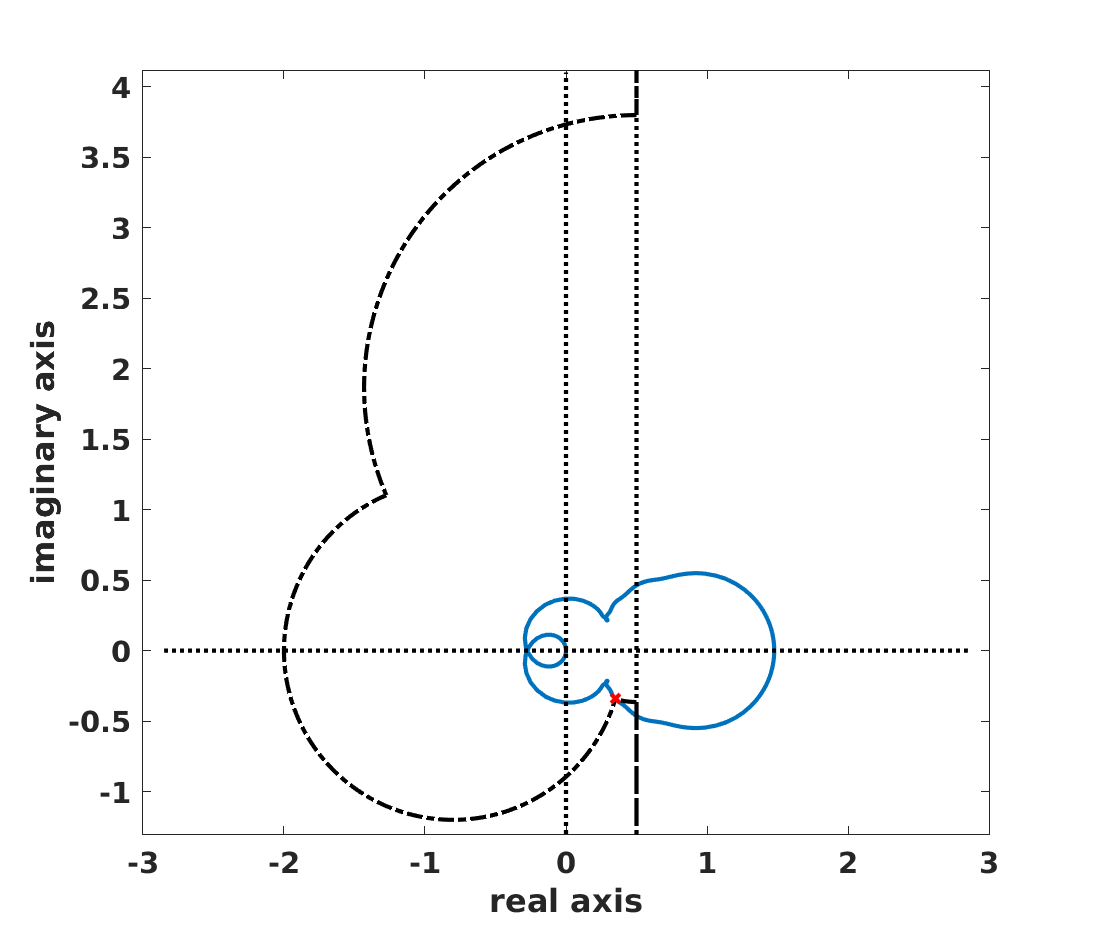}
 \subcaption{Frequency response plot of system \eqref{eq:T_controller_design_ex}, i.e., plot of $T(i\mR)$. The dash-dotted line segments show the boundary of the forbidden area at the critical frequency $\omega_c = 0.7188$, cf. Figure~\ref{subfig:interpolation_gain_phase_delay}, and the intersection between $T(i\mR)$ and the forbidden region is marked in red.}
 \label{subfig:ex_controller_design_T}
\end{subfigure}
\hfill
\begin{subfigure}[t]{.49\textwidth}
\centering
  \includegraphics[width=\textwidth]{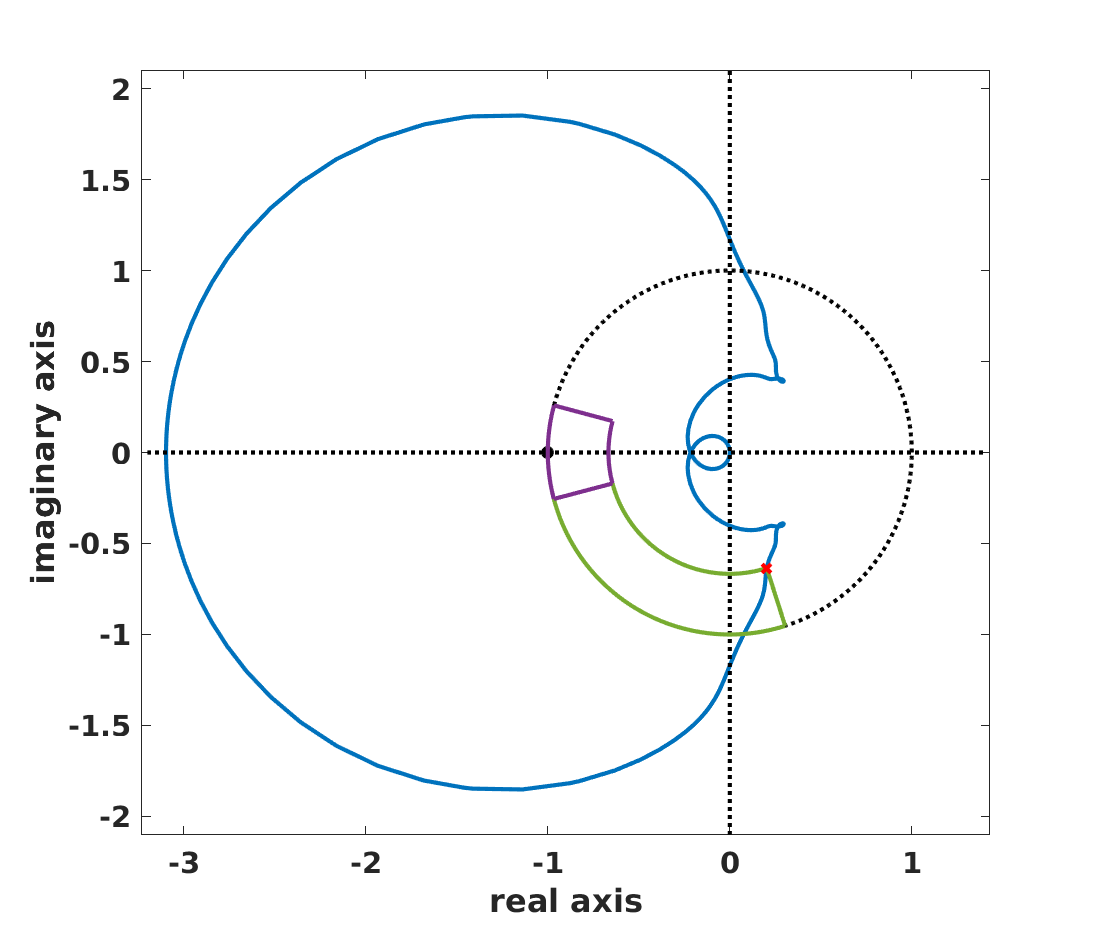}
 \subcaption{Nyquist plot of the system \eqref{eq:PK_controller_design_ex}. The region outline by the purple line segments corresponds to the simultaneous gain and phase margin of $1.5$ and $\pi/12$. The region outlined by the green line segments corresponds to adding the simultaneous delay margin $2.254$, illustrated for the critical frequency $\omega_c = 0.7188$, cf. Figure~\ref{subfig:nyquist_gain_phase_delay}. The point of intersection between the Nyquist plot and the forbidden region is marked in red.}
 \label{subfig:_controller_design_Nyquist}
 \end{subfigure}
\caption{Figures illustrating different aspects of system \eqref{eq:controller_design_ex}. Note that although it seems like both $T(i\mR)$ and the Nyquist plot enter the corresponding forbidden regions, and thus are both infeasible, this is not the case. Recall that the forbidden region is frequency dependent, and in each figure the forbidden region is only plotted for the frequency $\omega_c = 0.7188$.}
\label{fig:ex_controller_design}
\end{center}
\end{figure*}

\begin{figure*}[th]
%% Change text size
%\normalsize
%\small
%\footnotesize
\scriptsize
%\tiny
%
%% Store the current equation number.
%\newcounter{MYtempeqncnt}
%\setcounter{MYtempeqncnt}{\value{equation}}
%% Set the equation number to one less than the one
%% desired for the first equation here.
%% The value here will have to changed if equations
%% are added or removed prior to the place these
%% equations are referenced in the main text.
%\setcounter{equation}{5}
%% For placing on bottom of page (need to use stfloats-package as well)
%\vspace*{4pt}
\hrulefill
\begin{subequations}\label{eq:controller_design_ex}
\begin{align}
& T(s) =
\frac{  -9.476 s^{10} + 61.57 s^9 + 202.5 s^8 + 1083 s^7 + 1797 s^6 + 2943 s^5 + 2623 s^4 + 1880 s^3 + 855.6 s^2 + 235.5 s + 41.1}
{s^{11} + 37.79 s^{10} + 371.6 s^9 + 1267 s^8 + 3783 s^7 + 5502 s^6 + 7481 s^5 + 5352 s^4 + 3509 s^3 + 1147 s^2 + 303.6 s + 27.84}
\label{eq:T_controller_design_ex} \\
& P(s)K(s) =
\frac{-9.476 s^{10} + 61.57 s^9 + 202.5 s^8 + 1083 s^7 + 1797 s^6 + 2943 s^5 + 2623 s^4 + 1880 s^3 + 855.6 s^2 + 235.5 s + 41.1}
{s^{11} + 47.27 s^{10} + 310 s^9 + 1065 s^8 + 2700 s^7 + 3705 s^6 + 4538 s^5 + 2729 s^4 + 1629 s^3 + 291.1 s^2 + 68.13 s - 13.27}
\label{eq:PK_controller_design_ex}\\
& K(s) =
\frac{-94.76 s^{12} - 321.7 s^{11} - 1287 s^{10} - 2082 s^9 - 3392 s^8 - 3143 s^7 - 2689 s^6 - 1606 s^5 - 817.4 s^4 - 337.6 s^3 - 93.89 s^2 - 19.58 s - 1.257}
{0.1s^{12} + 4.744 s^{11} + 31.79 s^{10} + 111.6 s^9 + 287.7 s^8 + 415.3 s^7 + 515.3 s^6 + 348.2 s^5 + 208.1 s^4 + 56.13 s^3 + 11.64 s^2 - 0.1967 s - 0.2201}
%\frac{-947.6 s^{12} - 3217 s^{11} - 12870 s^{10} - 20820 s^9 - 33920 s^8 - 31430 s^7 - 26890 s^6 - 16060 s^5 - 8174 s^4 - 3376 s^3 - 938.9 s^2 - 195.8 s - 12.57}
%{s^{12} + 47.44 s^{11} + 317.9 s^{10} + 1116 s^9 + 2877 s^8 + 4153 s^7 + 5153 s^6 + 3482 s^5 + 2081 s^4 + 561.3 s^3 + 116.4 s^2 - 1.967 s - 2.201}
\label{eq:K_controller_design_ex}
\end{align}
\end{subequations}
%% Restore the current equation number.
%\setcounter{equation}{\value{MYtempeqncnt}}
%For placing on top of page
\hrulefill
\vspace*{4pt}
\end{figure*}

\section{Conclusions}\label{sec:conclusions}
In this work we have designed methods for determining a sharper lower bound for the maximum delay margin. This has been done in the context of the gain and phase margins, which are well-understood problems. The reason why the delay margin problem is much more difficult has been explained in terms of the Nyquist plot, where, unlike the situation for the gain and phase margins, the forbidden area is frequency dependent. Therefore we have introduced a nominal complementary sensitivity function to obtain a frequency-dependent shift in this function. We have designed a homotopy-based heuristic for selecting the nominal complementary sensitivity function. The problem is then solved by analytic interpolation and rational model reduction. We have also considered the simultaneous optimal delay margin problem under specifications on the gain and phase margins. The problem to determine the exact maximum delay margin is still unsolved. However, there are upper and lower bounds. In a number of numerical examples we have been able to establish sharper lower bounds, which in a few cases are seen to be essentially optimal as they coincide with upper bounds established in \cite{middleton2007achievable}.

\section*{Acknowledgement}
The authors would like to thank Jie Chen for introducing us to the maximum delay margin problem and for helpful discussions, and Leonid Mirkin for helpful discussions that ultimately lead us to the interpretation presented in  Section~\ref{subsec:small_gain_interp}. We also acknowledge Anders Blomqvist and Ryozo Nagamune for providing the code associated with \cite{blomqvist2005optimization}, which we used for solving the interpolation problems.

%%%%%%%%%%%%%%%%%%%%%%%%%%%%%%%%%%%%%%%%%%%%%%%%%%%%%%%%%%%%%%%%%%%%%%%%%%%%%%%%
\appendix

\subsection{Expressions related to simultaneous gain and phase margin}
\label{app:sim_gain_phase}

In this appendix we derive the expressions for the forbidden areas for the complementary sensitivity function $T(i\omega)$ given certain simultaneous gain and phase margin. Since the delay margin is a frequency-dependent phase margin, this can easily be incorporated in expressions below; cf. Figure~\ref{fig:nyquist_interpolation_gain_phase_delay_margin}.

To this end, in order to guarantee certain simultaneous gain and phase margin, consider Figures~\ref{subfig:nyquist_gain_phase_margin} and \ref{subfig:interpolation_gain_phase}, illustrating the forbidden areas for the Nyquist curve $K(i\omega)P(i\omega)$ and the complementary sensitivity function $T(i\omega)$, respectively. Expressions for the arcs and lines defining the forbidden area for $T$ are obtained by applying the M\"obius transformation \eqref{eq:mobius} to the arcs and lines defining the forbidden area for the Nyquist curve. Using \eqref{eq:mobius}, the arc coming from the smaller circle, defined by $(-1/k) \, e^{i\theta}$ for $\theta \in [-\varphi, \varphi]$, is transformed into the arc $1 / (1 - ke^{-i\theta}) = (1 - k\cos(\theta) - ik\sin(\theta))/(1 - 2k\cos(\theta) + k^2)$, $\theta \in [-\varphi, \varphi]$. The latter can be shown to be an arc belonging to a circle centered at $-1/(k^2 - 1)$ and with radius $k/(k^2 - 1)$. Moreover,  the two straight lines connecting the inner and outer arcs in Figure~\ref{subfig:nyquist_gain_phase_margin}, i.e., the lines $(-1/\rho) \, e^{\pm i\varphi}$ for $\rho \in [1, k]$, are similarly transformed to $1 / (1 - \rho e^{\mp i\varphi}) = (1 - \rho\cos(\varphi) \mp i\rho\sin(\varphi))/(1 - 2\rho\cos(\varphi) + \rho^2)$, $\rho \in [1, k]$. The latter corresponds to arcs belonging to circles centered at $1/2 \mp (i/2) \cot(\varphi)$ and having radius $1/|2\sin(\varphi)|$.

\subsection{Proof of Theorem~\ref{thm:dist_to_cut}}
\label{app:dist_to_cut}

The assumptions in the theorem are that $T \in \HinfC$, $T(\infty)$ is well-defined, and $\real(T(\infty)) < 1/2$. Therefore, for all results in this section it is assumed that these conditions hold.
\rev{The first lemma needed for the proof follows almost verbatim from the proof of the Nyquist stability criterion, see, e.g., \cite[Sec.~9.2]{astrom2008feedback}.}

\begin{lemma}\label{lem:encirclement}
For a fixed $\tau > 0$, $1 + T(s)(e^{-\tau s} - 1) \neq 0$ for all $s \in \mC_+$ if and only if the curve $\{ T(i \omega)(e^{-\tau i \omega} - 1) \mid \omega \in \mR \}$ does not encircle $-1$.
\end{lemma}

\begin{lemma}\label{lem:not_zero_s_omega}
Let $\bar{\tau} > 0$. The following two statements are equivalent:
\begin{enumerate}
\item[i)] $1 + T(s)(e^{-\tau s} - 1) \neq 0$ for all $s \in \bar{\mC}_+$ and $\tau \in [0, \bar{\tau}]$,
\item[ii)] $1 + T(i\omega)(e^{-\tau i\omega} - 1) \neq 0$ for all $\omega \in \mR$ and $\tau \in [0, \bar{\tau}]$.
\end{enumerate}
\end{lemma}

\begin{proof}
We proceed by proving the equivalence between the negation of the two statements, i.e., the equivalence between the two statements
\begin{enumerate}
\item[i\ensuremath{'})] there exists a $\tilde{\tau} \in [0, \bar{\tau}]$ and an $s \in \bar{\mC}_+$ such that $1 + T(s)(e^{-\tilde{\tau} s} - 1) = 0$,
\item[ii\ensuremath{'})] there exists a $\tilde{\tau} \in [0, \bar{\tau}]$ and an $\omega \in \mR$ such that $1 + T(i\omega)(e^{-\tilde{\tau} i\omega} - 1) = 0$.
\end{enumerate}
That the second statement implies the first is trivial, since $i\mR \subset \bar{\mC}_+$, and we can thus take the same point and the same $\tilde{\tau}$. Moreover, if in the first statement the point $s$ is such that $s \in i\mR$, then in same way the first statement implies the second. The lemma thus follows by showing that, if $i\ensuremath{'})$ holds for $s \in \mC_+$, then it implies $ii\ensuremath{'})$, and that neither $i\ensuremath{'})$ nor $ii\ensuremath{'})$ can hold for $s = \infty$, i.e., that there can be no sequence $(s_n)_n \subset \bar{\mC}_+$ such that $|s_n| \to \infty$ as $n \to \infty$ and such that
\[
\lim_{n \to \infty} 1 + T(s_n)(e^{-\tilde{\tau} s_n} - 1) = 0.
\]

To show the latter, note that since $\real(T(\infty)) < 1/2$ there exists an $\epsilon > 0$ and an $M > 0$ such that $\real(T(s)) < 1/2 - \epsilon$ for all $s \in \bar{\mC}_+$ such that $|s| > M$. Therefore, for $|s| > M$ we have that
\begin{align*}
&|1 + T(s)(e^{-\tau s} - 1)| =\\
&= \sqrt{\real((1 + T(s)(e^{-\tau s} - 1))^2 + \imag((1 + T(s)(e^{-\tau s} - 1))^2} \\
&\geq 1 - 2(1/2 - \epsilon)= 2\epsilon > 0.
\end{align*}
Therefore, such a sequence cannot exist.

In the case $i\ensuremath{'})$ holds for some $s \in \mC_+$, by Lemma~\ref{lem:encirclement} we have that for the same value $\tilde{\tau}$, the curve $\{ T(i \omega)(e^{-\tilde{\tau} i \omega} - 1) \mid \omega \in \mR \}$ encircles $-1$ at least once. However, all points on the curve are continuous in $\tau$, and for $\tau = 0$ the curve is identically $0$. Therefore, there must be a $\tau \in (0, \tilde{\tau})$ and an $\omega \in \mR$ such that $T(i \omega)(e^{-\tau i \omega} - 1) = -1$.
\end{proof}

To prove the theorem, we first prove that $1)$ implies $2)$. To this end, assume that there is an $\varepsilon > 0$ such that 
\[
|1 - T(s)(1 - e^{-s\tau})| \geq \varepsilon
\]
for all $s \in \bar{\mC}_+$ and all $\tau \in [0, \bar{\tau}]$. Observing that $1 - e^{-s \tau} \neq 0$ for all $s \in \mC_+$ and $\tau \in (0, \bar{\tau})$, this means that for all $s \in \mC_+$ and all $\tau \in (0, \bar{\tau})$
\[
\left| \frac{1}{1 - e^{-s\tau}} - T(s) \right| \geq \varepsilon \left| \frac{1}{1 - e^{-s\tau}} \right|.
\]
Therefore,
\begin{align*}
& \inf_{\omega \in \mR} \dist(\cut{\omega\bar{\tau}}, T(i\omega)) \! = \! \inf_{\omega \in \mR} \inf_{\tau \in (0, \bar{\tau}) } \left| \frac{1}{1 - e^{-i\omega \tau}} - T(i\omega) \right| \\
& \phantom{xxx} \geq \inf_{s \in \mC_+} \inf_{\tau \in (0, \bar{\tau}) } \left| \frac{1}{1 - e^{-s \tau}} - T(s) \right| \\
& \phantom{xxx} \geq \inf_{s \in \mC_+} \inf_{\tau \in (0, \bar{\tau}) } \varepsilon \left| \frac{1}{1 - e^{-s\tau}} \right| = \inf_{\tau \in (0, \bar{\tau}) } \inf_{s \in \mC_+} \varepsilon \left| \frac{1}{1 - e^{-s\tau}} \right|
\end{align*}
Next, noting i) that $1/(1 - e^{-s\tau}) = -e^{s\tau} /(1 - e^{s\tau})$, i.e., an application of the M\"obius transformation \eqref{eq:mobius} to $-e^{s\tau}$, and ii) that $\{ -e^{s\tau} \mid s \in \mC_+ \} = \bar{\mD}^C$, i.e, the outside of the closed unit disc, we have that $\{ 1/(1 - e^{-s\tau}) \mid s \in \mC_+ \} = \{ z \in \mC \mid \real(z) > 1/2 \}$. In particular, this means that irrespectively of $\tau$, $\inf_{s \in \mC_+} \varepsilon | 1/(1 - e^{-s\tau}) | \geq \varepsilon /2$. Therefore, if $|1 - T_0(s)(1 - e^{-s\tau})| \geq \varepsilon$ for all $s \in \bar{\mC}_+$ and all $\tau \in [0, \bar{\tau}]$, then
\[
\inf_{\omega \in \mR} \dist(\cut{\omega\bar{\tau}}, T(i\omega)) \geq \frac{\varepsilon}{2}.
\]

To prove the converse, assume there is an $\varepsilon > 0$ such that
\[
\inf_{\omega \in \mR} \dist(\cut{\omega\bar{\tau}}, T(i\omega)) \geq \varepsilon.
\]
In particular, this means that $| 1/(1 - e^{-i\omega \tau}) - T(i\omega) | \neq 0$ for all $\omega \in \mR$ and $\tau \in [0, \bar{\tau}]$, and thus that
\begin{equation}\label{eq:proof_thm_dist_cut}
| 1 - T(i\omega)(1 - e^{-i\omega \tau}) | \neq 0.
\end{equation}
By Lemma~\ref{lem:not_zero_s_omega}, this means that $1 + T(s)(e^{-\tau s} - 1) \neq 0$ for all $s \in \bar{\mC}_+$ and $\tau \in [0, \bar{\tau}]$. Left to show is that this also implies that
\[
\inf_{\substack{s \in \bar{\mC}_+ \\ \tau \in [0, \bar{\tau}]}} |1 + T(s)(e^{-\tau s} - 1)| > 0.
\]
To this end, assume that the latter is false. Then there is a sequence $(s_n, \tau_n)_n \subset \bar{\mC}_+ \times [0, \bar{\tau}]$ such that
\[
\lim_{n \to \infty} 1 + T(s_n)(e^{-\tau_n s_n} - 1) = 0.
\]
First note that since $[0, \bar{\tau}]$ is a bounded interval, there is a converging subsequence $\tau_n \to \tau_\infty \in [0, \bar{\tau}]$. Without loss of generality, we therefore consider only converging sequences $(\tau_n)_n$. Next, if the sequence $(s_n)_n$ is bounded, there is a subsequence converging to a point $s_\infty$. However, by continuity this means that $1 + T(s_\infty)(e^{-\tau_\infty s_\infty} - 1) = 0$, which contradicts \eqref{eq:proof_thm_dist_cut} and Lemma~\ref{lem:not_zero_s_omega}. Thus any such sequence $(s_n)_n$ must be unbounded. However, reexamining the proof of Lemma~\ref{lem:not_zero_s_omega} we see that due to the assumption $\real(T(\infty)) < 1/2$, there is an $\epsilon > 0$ and an $M > 0$ such that if $|s| \geq M$ then $|1 + T(s)(e^{-\tau s} - 1)| \geq 2\epsilon$, irrespectively of $\tau$. Therefore, such a sequence cannot exists, and the statement follows.
\qed

\subsection{Expression for $\absfuncwzero(\omega)$}
\label{app:phi_inv}

Since $1/ \sup_x f(x) = \inf_x 1/f(x)$ for nonnegative functions $f$, we have that
\begin{align*}
\absfuncwzero(\omega)^{-1} & = \frac{1}{\sup_{\tau \in [0, \bar{\tau}
]} \left| \frac{e^{-\tau i\omega} - 1}{1 - T_0(i\omega)(1 - e^{-\tau i\omega})} \right|} \\
& = \inf_{\tau \in [0, \bar{\tau}]} \left| \frac{1}{1 - e^{-\tau i \omega}} - T_0(i\omega) \right| \\
& = \dist(\cut{\omega\bar{\tau}}, T_0(i\omega)).
\end{align*}
This can be computed more explicitly, which will be done for frequencies $\omega \geq 0$. The expressions for $\omega < 0$ follow by a similar argument.

To this end, note that for $\omega > 0$ the forbidden region $\cut{\omega \bar{\tau}}$ is a cut entering from below, as illustrated in Figures~\ref{subfig:interpolation_delay_margin_small_omega} and \ref{subfig:interpolation_delay_margin_large_omega}. The expression for $\dist(\cut{\omega\bar{\tau}}, T_0(i\omega))$ thus depends on if $T_0(i\omega)$ has an imaginary part smaller than \rev{the} end-point of the cut, the imaginary part of which is given by $-0.5 \cot(\omega\bar{\tau}/2)$, or not. In the former case, the point on the cut closest to $T_0(i\omega)$ have the same imaginary part, and the distance is thus $|0.5 - \real(T_0(i\omega))|$. In the latter case, the point on the cut closest to $T_0(i\omega)$ is the the end-point. In this case the distance is $|0.5 - i0.5\cot({\omega \bar{\tau}/2}) - T_0(i\omega)|$.

Using the notation
\[
\eta_0^{\bar{\tau}}(\omega) := \imag(T_0(i\omega))  + \frac{1}{2}\cot(\omega \bar{\tau}/2),
\]
this can be summarized as
\begin{equation}\label{eq:absfuncwzero}
\absfuncwzero(\omega)^{-1} \! = \!
\begin{cases}
 |0.5 - \real(T_0(i\omega))|, \;\, \text{for } \omega \! \geq \! 2\pi/\bar{\tau}, \text{ or} \!\! \\
 \phantom{xxx} \text{for } 0 \! < \! \omega \! < \! 2\pi/\bar{\tau}  \text{ and } \eta_0^{\bar{\tau}}(\omega) \! \leq \! 0, \\
 \left|0.5 - i0.5\cot({\omega \bar{\tau}/2}) - T_0(i\omega) \right|,\\
 \phantom{xxx} \text{for } 0 \! < \! \omega \! < \! 2\pi/\bar{\tau} \text{ and } \eta_0^{\bar{\tau}}(\omega) \! > \! 0,\\
 \infty, \hspace{3pt} \text{for } \omega = 0, \\
 \left|0.5 + i0.5\cot({\omega \bar{\tau}/2}) - T_0(i\omega) \right|,\\
 \phantom{xxx} \text{for} \ -2\pi/\bar{\tau} \! < \! \omega \! < \! 0 \text{ and } \eta_0^{\bar{\tau}}(\omega) \! < \! 0 \!\! \\
 |0.5 - \real(T_0(i\omega))|, \;\, \text{for } \omega \! \leq \! -2\pi/\bar{\tau}, \, \text{or} \!\! \\
 \phantom{xxx} \text{for} \ -2\pi/\bar{\tau} \! < \! \omega \! < \! 0 \text{ and } \eta_0^{\bar{\tau}}(\omega) \! \geq \! 0. \!\!
\end{cases}
\end{equation}
Finally, in case $T_0(s) \equiv T_0$, constant, one can compute for which frequencies the transitions happen, see \cite[App.~B]{ringh2018lower}.

\subsection{Proof of Lemma~\ref{lem:T_0_infty}}
\label{app:proof_lem_T_0_infty}

Note that $(1 - T_0(s)(1 - e^{-s\tau}))^{-1} \in \HinfC$ for all $\tau \in [0, \bar{\tau}]$ if and only if there is an $\varepsilon > 0$ such that $|1 - T_0(s)(1 - e^{-s\tau})| \geq  \varepsilon$ for all $s \in \bar{\mC}_+$ and all $\tau \in [0, \bar{\tau}]$. Now, by Theorem~\ref{thm:dist_to_cut} this is true if and only if there is an $\varepsilon > 0$ such that $\dist(\cut{\omega \bar{\tau}}, T_0(i\omega)) \geq \varepsilon$ for all $\omega \in \mR$. To see that there is a $\bar{\tau} > 0$ such that the latter holds, note that if $\real(T_0(\infty)) < 1/2$, then there exists an $\varepsilon > 0$ and $M > 0$ such that for all $|\omega| \geq M$, $\real(T_0(i\omega)) \leq 1/2 - \varepsilon$ and thus $\inf_{\omega \in [-M, M]^C} \dist(\cut{\omega \bar{\tau}}, T_0(i\omega)) \geq \varepsilon$, irrespectively of $\bar{\tau}$. Now, consider only $\bar{\tau}$ such that $\bar{\tau} = \tilde{\tau} \pi/M$ for some $0 <\tilde{\tau} \leq 1$, and observe that $\cut{\omega \tilde{\tau} \pi/M}$ is contained in $[-\infty, 0]$ for $\omega \in[0, M]$ and in $[0, \infty]$ for $\omega \in [-M, 0]$ for all $0 < \tilde{\tau} \leq 1$. Moreover, set $\eta := \max_{\omega \in [-M, M]} | \imag(T_0(i\omega)) |$. If $|-0.5\cot(\omega\bar{\tau}/2)| \geq \eta + \varepsilon$ for all $\omega \in [-M, M]$, then by construction the result follows, since then $\cut{\omega\tilde{\tau}\pi/M}$ always have an imaginary part which in magnitude is at least $\varepsilon$ larger than that of $T_0(i\omega)$ on the interval $\omega \in [-M,M]$. Since the absolute value of $\cot$ is an even function, the problem is symmetric and thus we only consider $\omega \in [0, M]$. On this interval $|-0.5\cot(\omega\tilde{\tau}\pi/(2M))| = 0.5\cot(\omega\tilde{\tau}\pi/(2M))$, which is a nonincreasing function that attains its minimum for $\omega = M$. Thus, for any $\tilde{\tau} > 0$ such that $0.5\cot(\tilde{\tau}\pi/2) \geq \eta + \varepsilon$, i.e., for any $\tilde{\tau} > 0$ such that $\tilde{\tau} \leq (2/\pi) \cot^{-1}(2\eta + 2\varepsilon)$, and therefore for any $\bar{\tau} > 0$ such
\[
\bar{\tau} \leq \frac{2}{M} \cot^{-1}(2\eta + 2\varepsilon),
\]
$\dist(\cut{\omega \bar{\tau}}, T_0(i\omega)) \geq \varepsilon$ for all $\omega \in \mR$. This proves the first part of the lemma.

The second part of the lemma is proved by contradiction. To this end, assume that $\tau > 0$ and that there is an $\varepsilon > 0$ such that $|1 - T_0(s)(1 - e^{-s\tau})| \geq \varepsilon$ for all $s \in \mC_+$. Observing that $1 - e^{-s \tau} > 0$ for all $s \in \mC_+$, this means that for all $s \in \mC_+$
\begin{equation}\label{eq:lemma_real_part_1/2_proof}
\left| \frac{1}{1 - e^{-s\tau}} - T_0(s) \right| \geq \varepsilon \left| \frac{1}{1 - e^{-s\tau}} \right|.
\end{equation}
Next, by following the same reasoning as in the proof of Theorem~\ref{thm:dist_to_cut} (see Appendix~\ref{app:dist_to_cut}), the right-hand side of \eqref{eq:lemma_real_part_1/2_proof} is bounded from below by $\varepsilon/2$. Moreover, if $\real(T_0(\infty)) \geq 1/2$ then for any $\delta > 0$ we can find a point $s^* \in \mC_+$ such that $|\imag(s^*)| \leq \pi$ and $1/(1 - e^{-s^*\tau}) = T_0(\infty) + \delta$. Now, consider the sequence $(s^* + i2 \pi m)_m$. As $m \to \infty$, $T_0(s^* + i2 \pi m) \to T_0(\infty)$ and thus the right-hand side of \eqref{eq:lemma_real_part_1/2_proof} converges to $\delta$. This contradicts \eqref{eq:lemma_real_part_1/2_proof}, as $\delta \geq \varepsilon/2$ is not true for all $\delta > 0$ and fixed $\varepsilon > 0$. Since $\tau > 0$ was arbitrary, there can be no $\tau > 0$ such that $|1 - T_0(s)(1 - e^{-s\tau})| \geq \varepsilon$ for all $s \in \mC_+$ for some fixed $\varepsilon > 0$. This proves the lemma.
\qed

\subsection{Proof of Theorem~\ref{thm:approximation}}
\label{app:approximation}

We first prove the following lemma.

\begin{lemma}\label{lem:sigma_convergence}
Given $\sigma$ and $(\sigma_\ell)_{\ell = 1}^{\infty}$ as in Theorem~\ref{thm:approximation}, then $\sigma_\ell\to \sigma$ and $(\sigma_\ell)^{-1}\to (\sigma)^{-1}$ in $\HinfC$  as $\ell\to \infty$.
\end{lemma}

\begin{proof}
Without loss of generality we may assume that the phases of all $\sigma_\ell$ are shifted so that $\arg(\sigma_\ell(1))=\arg(\sigma(1))$. Next, from \cite[Thm.~2]{anderson1985continuity} the discrete time spectral factorization operator is continuous in $\Linf((-\pi, \pi])$ if there is a uniform bound on the $\Ltwo((-\pi, \pi])$-norm of the derivative of the spectral densities. Translating the domain via $z\to (1-z)/(1+z)$ into the continuous setting (cf. \cite[p.~122]{hoffman1962banach}) the corresponding condition is \eqref{eq:derbound}, i.e., a bound on the derivative in the weighted norm $\Ltwo(\mR,(1+\omega^2)d\omega)$. Thus, if the bounds \eqref{eq:derbound} holds, then we have that $\sigma_\ell\to \sigma$ in $\HinfC$  as $\ell\to \infty$. Moreover, since $|\sigma(i\omega)|$ is bounded away from zero on $\mR$, we have that $\sigma(s)$ is bounded away from zero on $\bar{\mC}_+$ and hence $(\sigma_\ell)^{-1}\to (\sigma)^{-1}$ in $\HinfC$.
\end{proof}

Note that $T=a/\sigma$ and $T_\ell=a_\ell/\sigma_\ell$ where $a,a_\ell \in \script{K}$ for all $\ell$. Moreover, since $T$ and $T_\ell$ satisfies \eqref{eq:opt2b}, $a$ and $a_\ell$ satisfy
\begin{subequations}\label{eq:a}
\begin{align}
a(s_k)&=w_k\sigma(s_k), \mbox{ for } k=0, \ldots, n, \label{eq:a1}\\
a_\ell(s_k)&=w_k\sigma_\ell(s_k),\mbox{ for } k=0, \ldots, n, \label{eq:a2}
\end{align}
\end{subequations}
for all $\ell$. The equations in \eqref{eq:a1} defines a system of equations
for the set of coefficients $(\alpha_k)_{k = 0}^{n}$ since $a$ is of the form \eqref{eq:coinvariant}. Since the equations are linearly independent they uniquely specify such $a$. Similarly, for each $\ell$, the set of equations \eqref{eq:a2} are linearly independent and thus uniquely specify  $a_\ell$ via the coefficients $(\alpha_k^\ell)_{k = 0}^{n}$. By Lemma~\ref{lem:sigma_convergence}, $\sigma_\ell \to \sigma$ in $\HinfC$ as $\ell \to \infty$, and in particular we have that $\sigma_\ell(s_k)\to \sigma(s_k)$ for $k=0,1,\ldots, n$. Therefore, the right hand side of \eqref{eq:a2} converges to that of \eqref{eq:a1} as $\ell \to \infty$, and hence $\alpha_k^\ell\to \alpha_k$ for $k=0,\ldots, n$. Thus, $a_\ell \to a$ in $\HinfC$ as $\ell \to \infty$. Finally, using this together with Lemma~\ref{lem:sigma_convergence} gives that $T_\ell=a_\ell/\sigma_\ell\to a/\sigma=T$ in $\HinfC$ as $\ell\to \infty$.
\qed

%%%%%%%%%%%%%%%%%%%%%%%%%%%%%%%%%%%%%%%%%%%%%%%%%%%%%%%%%%%%%%%%%%%%%%%%%%%%%%%%
% End of document; references, etc.                                            %
%%%%%%%%%%%%%%%%%%%%%%%%%%%%%%%%%%%%%%%%%%%%%%%%%%%%%%%%%%%%%%%%%%%%%%%%%%%%%%%%

\balance

\bibliographystyle{plain}
\bibliography{ref}

\begin{thebibliography}{10}

\bibitem{anderson1985continuity}
B.D.O. Anderson.
\newblock Continuity of the spectral factorization operation.
\newblock {\em Matema\'atica aplicada e computacional}, 4(2):139--156, 1985.

\bibitem{bjorstad1987conformal}
P.~Bj{\o}rstad and E.~Grosse.
\newblock Conformal mapping of circular arc polygons.
\newblock {\em SIAM Journal on Scientific and Statistical Computing},
  8(1):19--32, 1987.

\bibitem{blight1994practical}
J.D. Blight, R.L. Dailey, and D.~Gangsaas.
\newblock Practical control law design for aircraft using multivariable
  techniques.
\newblock {\em International Journal of Control}, 59(1):93--137, 1994.

\bibitem{blomqvist2005optimization}
A.~Blomqvist and R.~Nagamune.
\newblock Optimization-based computation of analytic interpolants of bounded
  complexity.
\newblock {\em Systems \& Control Letters}, 54(9):855--864, 2005.

\bibitem{byrnes2001ageneralized}
C.I. Byrnes, T.T. Georgiou, and A.~Lindquist.
\newblock A generalized entropy criterion for {Nevanlinna-Pick} interpolation
  with degree constraint.
\newblock {\em IEEE Transactions on Automatic Control}, 46(6):822--839, 2001.

\bibitem{doyle1992feedback}
J.C. Doyle, B.A. Francis, and A.~Tannenbaum.
\newblock {\em Feedback control theory}.
\newblock Macmillan, New York, N.Y., 1992.

\bibitem{dumitrescu2007positive}
B.~Dumitrescu.
\newblock {\em Positive Trigonometric Polynomials and Signal Processing
  Applications}.
\newblock Springer, Dordrecht, 2007.

\bibitem{fanizza2007passivity}
G.~Fanizza, J.~Karlsson, A.~Lindquist, and R.~Nagamune.
\newblock Passivity-preserving model reduction by analytic interpolation.
\newblock {\em Linear Algebra and its Applications}, 425(2-3):608--633, 2007.

\bibitem{foias1996robust}
C.~Foias, H.~{\"O}zbay, and A.~Tannenbaum.
\newblock {\em Robust control of infinite dimensional systems}.
\newblock Springer, Berlin, Heidelberg, 1996.

\bibitem{fridman2014introduction}
E.~Fridman.
\newblock {\em Introduction to time-delay systems}.
\newblock Birkh{\"a}user, Cham, 2014.

\bibitem{garnett2007bounded}
J.~Garnett.
\newblock {\em Bounded analytic functions}.
\newblock Springer, New York, N.Y., revised first edition, 2007.

\bibitem{gaudette2014stabilizing}
D.L. Gaudette and D.E. Miller.
\newblock Stabilizing a siso lti plant with gain and delay margins as large as
  desired.
\newblock {\em IEEE Transactions on Automatic Control}, 59(9):2324--2339, 2014.

\bibitem{georgiou2002structure}
T.T. Georgiou.
\newblock The structure of state covariances and its relation to the power
  spectrum of the input.
\newblock {\em IEEE Transactions on Automatic Control}, 47(7):1056--1066, 2002.

\bibitem{georgio2003kullback}
T.T. Georgiou and A.~Lindquist.
\newblock {K}ullback-{L}eibler approximation of spectral density functions.
\newblock {\em IEEE Transactions on Information Theory}, 49(11):2910--2917,
  2003.

\bibitem{glad2000control}
T.~Glad and L.~Ljung.
\newblock {\em Control Theory: Multivariable and Nonlinear Methods}.
\newblock Taylor \& Francis, New York, N.Y., 2000.

\bibitem{gu2003stability}
K.~Gu, J.~Chen, and V.L. Kharitonov.
\newblock {\em Stability of time-delay systems}.
\newblock Birkh\"auser, Boston, M.A., 2003.

\bibitem{helton1998classical}
J.W. Helton and O.~Merino.
\newblock {\em Classical Control Using {H}$^\infty$ Methods: Theory,
  Optimization, and Design}.
\newblock SIAM, Philadelphia, P.A., 1998.

\bibitem{hoffman1962banach}
K.~Hoffman.
\newblock {\em {B}anach Spaces of Analytic Functions}.
\newblock Prentice-Hall, Englewood Cliffs, N.J., 1962.

\bibitem{howell1993numerical}
L.H. Howell.
\newblock Numerical conformal mapping of circular arc polygons.
\newblock {\em Journal of computational and applied mathematics},
  46(1-2):7--28, 1993.

\bibitem{huang2000robust}
Y.-P. Huang and K.~Zhou.
\newblock Robust stability of uncertain time-delay systems.
\newblock {\em IEEE Transactions on Automatic Control}, 45(11):2169--2173,
  2000.

\bibitem{ju2016further}
P.~Ju and H.~Zhang.
\newblock Further results on the achievable delay margin using {LTI} control.
\newblock {\em IEEE Transactions on Automatic Control}, 61(10):3134--3139,
  2016.

\bibitem{ju2018achievable}
P.~Ju and H.~Zhang.
\newblock Achievable delay margin using {LTI} control for plants with unstable
  complex poles.
\newblock {\em Science China Information Sciences}, 61(9):092203, 2018.

\bibitem{kao2007stability}
C.-Y. Kao and A.~Rantzer.
\newblock Stability analysis of systems with uncertain time-varying delays.
\newblock {\em Automatica}, 43(6):959--970, 2007.

\bibitem{karlsson2010theinverse}
J.~Karlsson, T.T. Georgiou, and A.~Lindquist.
\newblock The inverse problem of analytic interpolation with degree constraint
  and weight selection for control synthesis.
\newblock {\em IEEE Transactions on Automatic Control}, 55(2):405--418, 2010.

\bibitem{karlsson2008stability-preserving}
J.~Karlsson and A.~Lindquist.
\newblock Stability-preserving rational approximation subject to interpolation
  constraints.
\newblock {\em IEEE Transactions on Automatic Control}, 53(7):1724--1730, 2008.

\bibitem{karlsson2009degree}
J.~Karlsson and A.~Lindquist.
\newblock On degree-constrained analytic interpolation with interpolation
  points close to the boundary.
\newblock {\em IEEE Transactions on Automatic Control}, 54(6):1412--1418, 2009.

\bibitem{khargonekar1985non}
P.~Khargonekar and A.~Tannenbaum.
\newblock Non-euclidian metrics and the robust stabilization of systems with
  parameter uncertainty.
\newblock {\em IEEE Transactions on Automatic Control}, 30(10):1005--1013,
  1985.

\bibitem{kimura1984robust}
H.~Kimura.
\newblock Robust stabilizability for a class of transfer functions.
\newblock {\em IEEE Transactions on Automatic Control}, 29(9):788--793, 1984.

\bibitem{luenberger1969optimization}
D.G. Luenberger.
\newblock {\em Optimization by Vector Space Methods}.
\newblock John Wiley \& Sons, New York, N.Y., 1969.

\bibitem{marshall2007convergence}
D.E. Marshall and S.~Rohde.
\newblock Convergence of a variant of the zipper algorithm for conformal
  mapping.
\newblock {\em SIAM Journal on Numerical Analysis}, 45(6):2577--2609, 2007.

\bibitem{megretski1997system}
A.~Megretski and A.~Rantzer.
\newblock System analysis via integral quadratic constraints.
\newblock {\em IEEE Transactions on Automatic Control}, 42(6):819--830, 1997.

\bibitem{michiels2002continuous}
W.~Michiels, K.~Engelborghs, P.~Vansevenant, and D.~Roose.
\newblock Continuous pole placement for delay equations.
\newblock {\em Automatica}, 38(5):747--761, 2002.

\bibitem{michiels2007stability}
W.~Michiels and S.-I. Niculescu.
\newblock {\em Stability and stabilization of time-delay systems: an
  eigenvalue-based approach}.
\newblock SIAM, Philadelphia, P.A., 2007.

\bibitem{middleton2007achievable}
R.H. Middleton and D.E. Miller.
\newblock On the achievable delay margin using {LTI} control for unstable
  plants.
\newblock {\em IEEE Transactions on Automatic Control}, 52(7):1194--1207, 2007.

\bibitem{miller2005stabilization}
D.E. Miller and D.E. Davison.
\newblock Stabilization in the presence of an uncertain arbitrarily large
  delay.
\newblock {\em IEEE Transactions on Automatic Control}, 50(8):1074--1089, 2005.

\bibitem{nehari1952conformal}
Z.~Nehari.
\newblock {\em Conformal mapping}.
\newblock McGraw-Hill, New York, N.Y., 1952.

\bibitem{qi2014fundamental}
T.~Qi, J.~Zhu, and J.~Chen.
\newblock Fundamental bounds on delay margin: When is a delay system
  stabilizable?
\newblock In {\em 33rd Chinese Control Conference (CCC)}, pages 6006--6013.
  IEEE, 2014.

\bibitem{qi2017fundamental}
T.~Qi, J.~Zhu, and J.~Chen.
\newblock Fundamental limits on uncertain delays: {W}hen is a delay system
  stabilizable by {LTI} controllers?
\newblock {\em IEEE Transactions on Automatic Control}, 62(3):1314--1328, 2017.

\bibitem{astrom2008feedback}
K.J. \r{A}str\"{o}m and R.M. Murray.
\newblock {\em Feedback systems}.
\newblock Princeton university press, Princeton, N.J., 2008.

\bibitem{ringh2018lower}
A.~Ringh, J.~Karlsson, and A.~Lindquist.
\newblock Lower bounds on the maximum delay margin by analytic interpolation.
\newblock In {\em 2018 IEEE Conference on Decision and Control (CDC)}, pages
  5463--5469. IEEE, 2018.

\bibitem{rudin1987real}
W.~Rudin.
\newblock {\em Real and {C}omplex {A}nalysis}.
\newblock McGraw-Hill, New York, 1987.

\bibitem{sepulchre1997constructive}
R.~Sepulchre, M.~Jankovi{\'c}, and P.V. Kokotovi{\'c}.
\newblock {\em Constructive nonlinear control}.
\newblock Springer, London, 1997.

\bibitem{tannenbaum1980feedback}
A.~Tannenbaum.
\newblock Feedback stabilization of linear dynamical plants with uncertainty in
  the gain factor.
\newblock {\em International Journal of Control}, 32(1):1--16, 1980.

\bibitem{trefethen1980numerical}
L.N Trefethen.
\newblock Numerical computation of the {S}chwarz--{C}hristoffel transformation.
\newblock {\em SIAM Journal on Scientific and Statistical Computing},
  1(1):82--102, 1980.

\bibitem{Walsh1956book}
J.L. Walsh.
\newblock {\em Interpolation and approximation by rational functions in the
  complex domain}.
\newblock American Mathematical Society, Providence, R.I., 2nd edition, 1956.

\bibitem{wang1994representation}
Z.-Q. Wang, P.~Lundstr{\"o}m, and S.~Skogestad.
\newblock Representation of uncertain time delays in the ${H}_{\infty}$
  framework.
\newblock {\em International Journal of Control}, 59(3):627--638, 1994.

\bibitem{youla1974single}
D.C. Youla, J.J. Bongiorno~Jr, and C.N. Lu.
\newblock Single-loop feedback-stabilization of linear multivariable dynamical
  plants.
\newblock {\em Automatica}, 10(2):159--173, 1974.

\bibitem{zames1983feedback}
G.~Zames and B.~Francis.
\newblock Feedback, minimax sensitivity, and optimal robustness.
\newblock {\em IEEE Transactions on Automatic Control}, 28(5):585--601, 1983.

\bibitem{zhou1996robust}
K.~Zhou, J.C. Doyle, and K.~Glover.
\newblock {\em Robust and optimal control}.
\newblock Prentice-Hall, Englewood Cliffs, N.J., 1996.

\end{thebibliography}

\end{document}